\documentclass[12pt,oneside,a4papper]{osajnl}
\usepackage{amsthm}
\usepackage{comment}
\journal{jocn} 

\setboolean{shortarticle}{false}

\title{\centering 
Decay of solutions and bilinear control to trajectories of a 1D degenerate parabolic system}

\author[1,2,*]{\centering Alfredo S. Gamboa}
\author[3]{André da Rocha Lopes}
\author[4]{Luis P. Yapu}

\affil[1]{Universidade do Estado do Rio de Janeiro, Escola Politécnica, Nova Friburgo, Brazil}
\affil[2]{Universidad Privada Boliviana, Departamento de Ciencias Exactas, Cochabamba, Bolivia}
\affil[3]{Universidade do Estado do Rio de Janeiro, Instituto de Matemática e Estatística, Rio de Janeiro, Brazil}
\affil[4]{Universidade Federal Fluminense, Instituto de Matemática e Estatística, Niterói, Brazil}

\affil[*]{\centering Contato: alfredo.soliz@iprj.uerj.br}

 \dates{\centering Compiled \today}

\doi{}

\begin{abstract}
This paper is concerned with the decay of solutions and the analysis of the local and global exact controllability to trajectories for a class of one-dimensional nonlinear parabolic systems with weakly degenerate diffusion coefficients. The system is controlled through the coefficient of the reaction term. Our approach relies on a well-known local inversion method combined with suitable a priori estimates specifically adapted to the degenerate setting.

\end{abstract}

\definecolor{mygreen}{RGB}{44,162,67}
\definecolor{mylilas}{RGB}{186,85,211}
\setboolean{displaycopyright}{false}

\newcommand{\cara}{\mathbb{1}}

\newtheorem{teo}{Theorem}
\newtheorem{myth}{Theorem}
\newtheorem{mydef}{Definition}
\newtheorem{propo}{Proposition}
\newtheorem{coro}{Corollary}

\newcommand{\R}{\mathbb{R}}

\setboolean{displaycopyright}{false}

\begin{document}

\maketitle

\textbf{MSC Classification (2020)}: Primary: 35K65, 93B05; Secondary: 93C10. 

\textbf{keywords}: Degenerate parabolic system, Nonlinear systems in Control Theory, Bilinear control, Carleman inequalities.

\section*{Introduction}

\qquad In the present work, we address the controllability properties for the following bilinear degenerate parabolic system:
\begin{equation}\label{eq:PDE}
\left\{
\begin{array}
    [c]{lll}%
    y_t - b(t)\left(a(x) y_{x}\right)_{x} + B_1 (x,t) \sqrt{a} y_x +F_1(x,t,y,z)=h\cara_{\omega}y & (x,t) \in (0,1)\times (0,T),&
    \\
    z_t - b(t)\left(a(x) z_{x}\right)_{x} + B_2 (x,t) \sqrt{a} z_x+ F_2(x,t,y,z)=0 & (x,t) \in (0,1)\times (0,T),&
    \\
    y(0,t)=y(1,t)=z(0,t)=z(1,t)=0 & t \in (0,T), & \\
    y(x,0)=y_{0}(x), \ z(x,0)=z_{0}(x) & x \in (0,1), & 
\end{array}
\right.  %
\end{equation}
where, $(y_0 ,z_0) \in [H_{a}^1 (0,1)]^2$, $T>0$ is fixed and $\cara_{\omega}$ stands for a characteristic function of a nonempty open subset $\omega$ of $(0,1)$. Here, $(y,z)$ denotes the state variable, while $h$ represents the control function. The coefficient $a(x)$ is a diffusion parameter that degenerates at the boundary point $x=0$. Such degeneracy characterizes diffusion processes in heterogeneous media where the transport mechanism becomes progressively weaker near a critical region of the spatial domain. A representative example is provided by the coefficient $a(x)=x^\alpha$, with $0<\alpha<1$, which models a medium whose diffusivity decreases in the vicinity of $x=0$. Structures of this type naturally arise in thermal processes in nonhomogeneous materials, sedimentation phenomena, and diffusion in anisotropic media.

The non-autonomous character of system (\ref{eq:PDE}) reflects environments whose physical properties evolve over time. Indeed, the coefficient $b(t)$ may represent time-dependent conductivity, external forcing, or seasonal effects, while the lower-order terms may describe convection, interaction, or reaction mechanisms depending on the application under consideration.

Throughout the whole paper, the following hypotheses will be assumed:
\begin{itemize}
    \item[{\textbf {H1}}] $a \in C([0,1]) \cap C^1 ((0,1])$ satisfying $a(0)=0$, $a>0$ on $(0,1]$, $a' \geq 0$ and $$xa'(x)\leq Ka(x),\,\forall x \in [0,1]\,\,\text{and some}\,\,K \in [0,1).$$
    
    \item[{\textbf {H2}}] For $i=1,2$, $F_i$ are $C^1$ functions, with bounded derivatives, satisfying $F_i(x,t,0,0)=0$ and 
    \begin{equation*}\label{hf}
        \begin{split}
            \vert F_i(x,t,r_1,s_1)-F_i(x,t,r_2,s_2)- \partial_3 F_i(x,t,r_2,s_2)(r_1 -r_2) - \partial_4 F_i(x,t,r_2,s_2)(s_1 -s_2)\vert \\
            \qquad \leq C (\vert r_1 - r_2\vert^2 + \vert s_1 - s_2\vert^2) \,\, \text{for any} \,\, (r_1 ,r_2,s_1,s_2) \in \mathbb{R} \times \mathbb{R} \times \mathbb{R} \times \mathbb{R}.        
        \end{split}
    \end{equation*}
    We also consider $d_{ij}:=\partial_{j+2}F_i (x,t,0,0) \in L^{\infty}((0,1) \times (0,T))$, for any $i,j \in \{1,2\}$ and let us assume that there exists $\omega_0 \subset \subset \omega$ such that
    $$\inf \{ a_{21}: (x,t) \in \omega_0 \times [0,T]\}>0\,.$$
    
    \item[{\textbf {H3}}] The coefficients $B_i (x,t),\,(i=1,2)$ are bounded and the time dependent function $b(t) \in W^{1,\infty} (0,T)$ and satisfy $b(t)>b_0 $, for some $b_0 > 0$. We also suppose that $B_1 (x,t) \sqrt{a}$ is of class $C^1$.
    
    
\end{itemize}

In order to deal with the controllability properties of problem (\ref{eq:PDE}), it is necessary to introduce the following Sobolev weighted spaces
\begin{equation*}
    \begin{split}
        H^1_a(0,1):=&\left\{ w\in L^2(0,1) \ : \ w \ \text{is absolutely continuous in} \ \ (0,1], \right. \\
        &\left. \quad \sqrt{a}w_{x}\in L^2(0,1) \ \text{and} \  w(0)=w(1)=0 \right\}        
    \end{split}
\end{equation*}
and 
$$H^2_a(0,1):=\left\{w\in H^1_a(0,1) \ : \ aw_{x}\in H^1(0,1) \right\},$$
with respective norms
$$
\Vert w\Vert^{2}_{H^1_a(0,1)}:= \Vert w\Vert^{2}_{L^2(0,1)} + \Vert \sqrt{a}w_{x}\Vert^{2}_{L^2(0,1)} \,\, \text{and} \,\,\Vert w\Vert^{2}_{H^2_a(0,1)}:= \Vert w\Vert^{2}_{H^1_a(0,1)} + \Vert (aw_{x})_{x}\Vert^{2}_{L^2(0,1)}\,.
$$

The notion of local controllability that we consider in this work is defined as follows.
\begin{mydef}
    It is said that (\ref{eq:PDE}) is null controllable at time $T$ if, for any $(y_0 ,z_0 ) \in [H_{a}^{1}(0,1)]^2$ there exists a control function $h \in L^2 (\omega \times (0,T))$ such that the associated state $(y,z)$ satisfies
    \begin{equation}\label{nc}
    y(\cdot ,T)=z(\cdot ,T)=0 \,\,\,\text{in}\,\,\, (0,1)\,.
    \end{equation}
\end{mydef}

The goal of this work is to study the controllability properties of (\ref{eq:PDE}) in the following sense: We consider a trajectory $(\widetilde{u}, \widetilde{v}) \in [L^2 (0,T;H^{2}_{a}(0,1))]^2$ satisfying the uncontrolled system
\begin{equation}\label{traj.prob1}
    \left\{
    \begin{array}[c]{lll}%
        \widetilde{y}_{t}- b(t)\left(a(x) \widetilde{y}_{x}\right)_{x} + B_1 (x,t)\sqrt{a} \widetilde{y}_x +F_1(x,t,\widetilde{y},\widetilde{z})=0 & (x,t) \in (0,1)\times (0,T), &
        \\
        \widetilde{z}_{t}- b(t)\left(a(x) \widetilde{z}_{x}\right)_{x} + B_2 (x,t)\sqrt{a} \widetilde{z}_x+F_2(x,t,\widetilde{y},\widetilde{z})=0 & (x,t) \in (0,1)\times (0,T), &
        \\
        \widetilde{y}(0,t)=\widetilde{y}(1,t)=\widetilde{z}(0,t)=\widetilde{z}(1,t)=0 & t \in (0,T), & \\
        \widetilde{y}(x,0)=\widetilde{y}_{0}(x), \ \widetilde{z}(x,0)=\widetilde{z}_{0}(x)  & x \in (0,1). & 
    \end{array}
    \right.  %
\end{equation}

\begin{mydef}
    It is said that (\ref{eq:PDE}) is exactly controllable to the trajectory $(\widetilde{y}, \widetilde{z})$ at time $T$ if, for any $(y_0 ,z_0)\in [H_{a}^{1}(0,1)]^2$ there exists a control function $h \in L^2 (\omega \times (0,T))$ such that the associated state $(y,z)$ satisfies
    \begin{equation}\label{ct}
    y(\cdot ,T)= \widetilde{y}(\cdot ,T) \,\,\,\text{and}\,\,\, z(\cdot ,T)= \widetilde{z}(\cdot ,T)\,\,\,\text{in}\,\,\, (0,1) \,.
    \end{equation}
\end{mydef}

Control of degenerate parabolic equations is a fairly well-developed subject in Control Theory. It is important to remark that semilinear nondegenerate equations have been studied extensively in the last decades, see \cite{dcbz-02, fr-71,clm-12,cz-00,FurImanu-96,lr-95} in the context of bounded cylindrical domains and \cite{cllp-25,lcmml-16,ll-22} in more general domains. Moreover, in the context of degenerate reaction-diffusion equations there are several interesting models, such as models in mathematical biology and in a wide variety of physical situations, see for instance \cite{cpz-04,kalash-87,nss-00}. In recent years several contributions treating degenerate PDEs appeared, in particular we mention the works by Cannarsa and collaborators \cite{Alabau_cannarsa_fragnelli-06,abcl-17,cmv1-04,cmv2-08,cmv3-16,cmv4-17,cty-10}.

 The study of controllability for non-autonomous degenerate coupled parabolic equations is driven by real-world phenomena where diffusion rates vanish at boundaries or critical points, and the system's underlying parameters or environmental factors change over time. Some recent contributions concerning the controllability of degenerate non-autonomous equations are the ones of Akil, Fragnelli and Ismail \cite{AkilFragnelliIsmail2025} for Robin boundary conditions, \cite{fragneliautonomo}, where a climatological application is presented, and a more general result for Dirichlet and Robin boundary conditions \cite{fragnelli-general-26}. On the other hand, it is worth mentioning that, in \cite{GYL-Carleman-2025}, the authors established a local null controllability result for the following linear non-autonomous degenerate coupled system:
\begin{equation}\label{ajl}
    \begin{cases}
		y_{t}-b(t)\left({a}(x)y_x\right)_x+ d_1 (x,t)\sqrt{a}y_{x}+b_{11}(x,t)y\\
        \qquad + b_{12}(x,t)z=h \cara_{_{{\omega}}} + H_1, & \ \ \  (x,t) \in (0,1)\times (0,T),\\
        z_{t}-b(t)\left({a}(x)z_x\right)_x+ d_2 (x,t)\sqrt{a}z_{x}+d_{21}(x,t)y\\
        \qquad + d_{22}(x,t)z=H_2, & \ \ \  (x,t) \in (0,1)\times (0,T),\\
		y(0,t)=y(1,t)=z(0,t)=z(1,t)=0 & \ \ \ t \in (0,T),\\
		 y(x,0)=y_{0}(x), \ z(x,0)=z_{0}(x) & \ \ \ x \in (0,1).
    \end{cases}
\end{equation}
where the coefficients $d_i(x,t)$ and $b_{ij}(x,t),\,(i,j=1,2)$ are bounded and $H_i \in L^2((0,1)\times (0,T)), \,(i=1,2)$. In the present paper, motivated by possible extensions to other relevant and realistic problems, we improve the results obtained in \cite{GYL-Carleman-2025} by considering a multiplicative control $h\cara y$, incorporating semilinear terms $F_i(x,t,y,z),\,(i=1,2)$, and establishing decay estimates for the solutions, as well as local and global exact controllability to trajectories for system (\ref{eq:PDE}).

\color{black}
The presence of a bilinear control that acts through the reaction term represents a realistic mechanism for influencing the system. In contrast to additive controls, the control considered here acts proportionally to the current state, modulating growth or decay rates. This type of control appears, for example, in chemical reactions where catalysts affect reaction rates, in ecological models where reproduction rates are regulated, or in thermal systems with feedback-dependent dissipation.

From a practical standpoint, the controllability properties studied in this work correspond to the ability to steer the system toward a desired configuration. Null controllability is associated with the suppression or extinction of a physical quantity, while controllability to a positive trajectory reflects the possibility of tracking a prescribed evolution. These objectives are of significant interest in applications where one seeks either to stabilize or to regulate complex dynamical systems under realistic constraints.

The goal of this work is to give some results on the bilinear controllability of a degenerate parabolic system. We refer to the early paper \cite{bms-82} on controllability of an abstract infinite dimensional bilinear system, which appears to be the first work on this subject in the framework of PDEs. In \cite{khapalov-02}, the author discussed the non-negative approximate controllability of a parabolic system with superlinear term governed by a bilinear control. Moreover, in \cite{khapalov_newton-02} he also discussed the bilinear null-controllability of a parabolic system with the reaction term satisfying Newton's Law. We also refer to the article \cite{pzh-06}, on exact controllability of parabolic systems. Important progress has been made recently in the 
analysis of bilinear controllability of parabolic equations, we cite, for instance, Alabau-Boussouira et al. \cite{abcu-22,abcu-24} and Cannarsa et al. \cite{cdu-22}. In the context of degenerate hyperbolic equations, we mention Cannarsa et al. \cite{cmu-23}.



The main results in this paper are as follows.
\begin{myth}
    \label{th1}
    Under the previous assumptions on the functions $a$ and $F_i , (i=1,2)$, the nonlinear system (\ref{eq:PDE}) is locally exactly controllable to the trajectory $(\widetilde{y} ,\widetilde{z})$ at any time $T>0$. In other words, there exists $\varepsilon >0$ such that, for any $(y_0,z_0) \in [H_{a}^{1}(0,1)]^2$ with $$\Vert (y_0 - \widetilde{y}_0 ,z_0 - \widetilde{z}_0) \Vert_{[H_{a}^{1}(0,1)]^2} \leq \varepsilon\,,$$ there exists a control $h \in L^2 (\omega \times (0,T))$ such that the associated state $(y,z)$ satisfies (\ref{ct}).
\end{myth}

\begin{myth}
\label{th2}
Under the previous assumptions on the functions $a$ and $F_i , (i=1,2)$, the nonlinear system (\ref{eq:PDE}) is globally exactly controllable to the trajectory $(\widetilde{y} ,\widetilde{z})$ at any sufficiently large time, that is, there exists $\widetilde{T}>0$ such that for any $(y_0,z_0) \in [H_{a}^{1}(0,1)]^2$ and $T>\widetilde{T}$, there exists a control $h \in L^{2}(\omega \times (0,T))$ such that the associated state $(y,z)$ satisfies (\ref{ct}).
\end{myth}

The strategy to prove Theorem \ref{th1} relies on an application of the \emph{Liusternik's Inverse Function Theorem} in Banach spaces; see \cite{Alekseev}. To this end, we will check that the assumptions of Liusternik's Theorem are satisfied and, consequently for any small initial data $(y_0 ,z_0)$, system (\ref{eq:PDE}) is solvable.

To prove Theorem \ref{th2}, we will show that for every initial data $(y_0 ,z_0) \in [H^{1}_{a}(0,1)]^2$, the solutions of the system (\ref{eq:PDE}) without control, that is $h\equiv 0$, have an exponential decay in time. In this way, for $\widetilde{T}$ large enough $(y(\cdot ,\widetilde{T}), z(\cdot ,\widetilde{T}))$ belongs to a small ball of $[H^{1}_{0}(0,1)]^2$ centered at $(0,0)$, where by Theorem \ref{th1}, the local exact controllability to trajectories holds. Thus, we get $(y(\cdot ,T), z(\cdot ,T))=(\widetilde{y}(\cdot ,T), \widetilde{z}(\cdot ,T))$ with an appropriate choice of control.

The rest of the paper is structured as follows. In Section \ref{sec:trajectory}, we give the details of the change of variables and prove a result of decay of solutions. In Section \ref{sec:linearized_system}, we consider and solve a null controllability problem for an associated linear parabolic system; this will be needed later to prove that the hypotheses of Liusternik's Theorem are fulfilled. Section \ref{sec:control for nonlinear system} deals with the proof of Theorem \ref{th1} and Theorem \ref{th2}. Finally, some additional comments are presented in Section \ref{sec:final_remarks}.

\section{Reformulation of the problem}
\label{sec:trajectory}
\qquad Let us introduce a nonempty open set $\omega_0 \subset \subset \omega$ and trajectories $\widetilde{y}=\widetilde{y}(x,t)$ and $\widetilde{z}=\widetilde{z} (x,t)$ solutions of (\ref{eq:PDE}), such that
$$
\vert \widetilde{y}\vert >C >0 \quad \text{in} \quad \omega_0 \times (0,T)\,.
$$

Note that, if we set $y=w_1+\widetilde{y},$ \ $z=w_2+\widetilde{z}$,  $y_0 =w_{1,0} +\widetilde{y}_{0}$ and $z_0 =w_{2,0} +\widetilde{z}_{0}$ , we get the system
\begin{equation}\label{nolineal}
    \begin{cases}
		w_{1,t}-b(t)\left({a}(x)w_{1,x}\right)_x-B_1(x,t)\sqrt{a}w_{1,x}+F_1\left(x,t,w_1+\widetilde{y},w_2+\widetilde{z}\right) \\
        \qquad -F_1\left(x,t,\widetilde{y},\widetilde{z}\right)={h}\cara_{_{{\omega}}}(w_1+\widetilde{y}), & \ \ \ \text{in} \ \ \ {Q},\\
        w_{2,t}-b(t)\left({a}(x)w_{2,x}\right)_x-B_2(x,t)\sqrt{a}w_{2,x}+F_2\left(x,t,w_1+\widetilde{y},w_2+\widetilde{z}\right) \\
        \qquad - F_2\left(x,t,\widetilde{y},\widetilde{z}\right)=0, & \ \ \ \text{in} \ \ \ {Q},\\
		w_1=0, \ \ w_2=0 & \ \ \ \text{on} \ \ \ {\Sigma},\\
		w_1(0)=y_0-\widetilde{y}_0=w_{1,0}, \ \ w_2(0)=z_0-\widetilde{z}_0=w_{2,0}, & \ \ \ \text{in} \ \ \ (0,1),
    \end{cases}
\end{equation}
where, for simplicity, we represent $Q:=(0,1)\times (0,T)$ and $\Sigma := \partial (0,1) \times (0,T)$.

Thus, the local exact controllability of the solution to (\ref{eq:PDE}) to $(\widetilde{y},\widetilde{z})$ is equivalent to the null controllability of the solution to (\ref{nolineal}).
The null controllability problem for (\ref{nolineal}) can be formulated as an inversion local method in the Banach Spaces $\mathsf{V}$ and $\mathsf{W}$. To this end, let us consider the mapping $\mathcal{A} : \mathsf{V} \to \mathsf{W}$, given by
$$
\mathcal{A} (w_1,w_2,h)=(\mathcal{A}_{1,1}(w_1,w_2,h),\mathcal{A}_{1,2}(w_1,w_2,h),w_1(\cdot,0),w_2(\cdot,0)), 
$$
where
\begin{align*}
\mathcal{A}_{1,1}(w_1,w_2,h)  & =w_{1,t}-b(t)\left({a}(x)w_{1,x}\right)_x-B_1(x,t)\sqrt{a}w_{1,x}+F_1\left(x,t,w_1+\widetilde{y},w_2+\widetilde{z}\right)\\
 & \quad -F_1\left(x,t,\widetilde{y},\widetilde{z}\right)-{h}\cara_{_{{\omega}}}(w_1+\widetilde{y}),\\
\mathcal{A}_{1,2}(w_1,w_2,h) &=w_{2,t}-b(t)\left({a}(x)w_{2,x}\right)_x-B_2(x,t)\sqrt{a}w_{2,x}+F_2\left(x,t,w_1+\widetilde{y},w_2+\widetilde{z}\right) \\
  & \quad  -F_2\left(x,t,\widetilde{y},\widetilde{z}\right).
\end{align*}

Later, our approach will be to locally invert this mapping, with the aim to ensure, for any small initial data $(w_{1,0},w_{2,0}) \in [H_{a}^1 (0,1)]^2$, the existence of triplets $(w_1,w_2,h) \in \mathsf{V}$ such that $\mathcal{A}(w_1,w_2,h) =(G_1 ,G_2 ,w_{1,0},w_{2,0})$ for $(G_1 ,G_2 ,w_{1,0},w_{2,0}) \in \mathsf{W}$.

To accomplish this goal, we will apply the following local inversion result in Banach spaces, which is a consequence of the so-called Liusternik-Graves Theorem:
\begin{teo}[Liusternik \cite{Alekseev}]\label{Liusternik} 
    Let $ \mathsf{V}$ and $ \mathsf{W}$ be Banach spaces and let $\mathcal{A}:B_{r}(0)\subset  \mathsf{V}\rightarrow  \mathsf{W}$ be a $\mathcal{C}^{1}$ mapping. Let us assume that $\mathcal{A}^{\prime}(0)$ is onto and let us set $\mathcal{A}(0)=\zeta_{0}$. Then, there exist $\delta >0$, a mapping $\widetilde{\mathcal{A}}: B_{\delta}(\zeta_{0})\subset  \mathsf{W}\rightarrow  \mathsf{V}$ and a constant $K>0$ such that
    \begin{equation*}
        \widetilde{\mathcal{A}}(z)\in B_{r}(0),\,\, \mathcal{A}(\widetilde{\mathcal{A}}(z))=z\,\, \text{and}\,\, \Vert \widetilde{\mathcal{A}}(z)\Vert_{ \mathsf{W}}\leq K\Vert z-\zeta_{0}\Vert_{ \mathsf{W}}\, \, \forall\, z\in B_{\delta}(\zeta_{0}).
    \end{equation*}
    In particular, $\widetilde{\mathcal{A}}$ is a local inverse-to-the-right of $\mathcal{A}$.
\end{teo}

Once the inversion is performed, we take $(w_1,w_2,h)=\mathcal{A}^{-1}(G_1 ,G_2 ,w_{1,0},w_{2,0})$ and we see that the functions $y=y(x,t)$, $z=z(x,t)$ satisfy (\ref{ct}).

\subsection{Well-posedness and decay of solutions} \label{wpds}
\quad In this section, we will state well-posedness and decay of solutions for the nonlinear system (\ref{nolineal}) without control $h$. Such method is crucial to analyze global null controllability of (\ref{nolineal}). Precisely, let us consider the function $F_i \,\,(i=1,2)$ as in the assumption {\textbf {(H2)}}. Additionally, let us suppose that
\begin{equation}
    \label{cond_for_decay}
    r_1 (F_1\left(x,t,r_1+\widetilde{y},r_2+\widetilde{z}\right)-F_1 (x,t,\widetilde{y},\widetilde{z})) + r_2 (F_2\left(x,t,r_1+\widetilde{y},r_2+\widetilde{z}\right)-F_2 (x,t,\widetilde{y},\widetilde{z}))  \geq 0,      
\end{equation}
for any $(r_1,r_2) \in \mathbb{R} \times \mathbb{R}$.

For this, we consider the following result:
\begin{teo}\label{twpds}
    Let  $(w_{1,0},w_{2,0}) \in [H_{a}^1 (0,1)]^2$  and $h=0$. Under the assumptions on $F_i \,\,(i=1,2)$, the system (\ref{nolineal}) has a unique solution $(w_1,w_2)$ in the class
     $$(w_1,w_2) \in L^2(0,T;L^2(0,1))\cap L^{\infty}(0,T;L^2(0,1)).$$
    Moreover, supposing in addition that (\ref{cond_for_decay}) holds, 
    there exist positive constants $C$ and $\mu$, independent of $t$, such that
    \begin{equation} \label{cl2}
        \Vert w_1 (\cdot,t)\Vert^{2}_{L^2(0,1)}+\Vert w_2 (\cdot,t)\Vert^{2}_{L^2(0,1)}+\Vert \sqrt{a} \,w_{1,x} (\cdot,t)\Vert_{L^2(0,1)}^{2}+\Vert  \sqrt{a}\, w_{2,x}(\cdot,t)\Vert_{L^2(0,1)}^{2} \leq C e^{-\mu t}, \,\,\forall t \geq 0.
    \end{equation}
\end{teo}

\begin{proof}
    We observe that the non-linear system (\ref{nolineal}) has the form treated in Appendix A of \cite{GYL-Carleman-2025} when $h=0$. We sketch the proof here and obtain the exponential decay.
    
    Consider the system of semi-linear equations: 
    \begin{equation}\label{PDE_annexe}
    	\left\{\begin{aligned}
    		&u_t - b(t) \left(a(x) u_{x}\right)_{x} + d_1(x,t)\sqrt{a}u_x + \tilde F_1(x,t,u,v) = 
            H_1 &&\text{in} && Q,\\
    		&v_t - b(t) \left(a(x) v_{x}\right)_{x} + d_2(x,t)\sqrt{a}v_x + \tilde F_2(x,t,u,v) = H_2 &&\text{in}&& Q,\\
    		&u(0,t)=u(1,t)=v(0,t)=v(1,t)=0&&\text{on} && (0,T), \\
    		&u(\cdot,0) = u_0, \quad v(\cdot,0) = v_0 &&\text{in} && (0,1),
    	\end{aligned}
    	\right.
    \end{equation}
    where $d_1$, $d_2$ are bounded functions on $Q$, $0 < b_0 \leq b(t)$ bounded
    and $\tilde F_1$ and $\tilde F_2$ are globally Lipschitz in the third and fourth entries. To get system (\ref{nolineal}), we take $u=w_1$ and $v=w_2$ and, for $i=1,2$,
    $$
        \tilde F_i(x,t,u,v) = F_i\left(x,t,u+\widetilde{y},v+\widetilde{z}\right)-F_i (x,t,\widetilde{y},\widetilde{z}).
    $$

    Let $(\mathsf{w}_{i})_{i}^{\infty}$ be an orthonormal basis of $H^{1}_{a}(0,1)$ such that $-(a(x)\mathsf w_{i,x})_{x}=\lambda_{i}\mathsf w_{i}$. Thus,
    \begin{equation*}
    	-b(t)(a(x) \mathsf w_{i,x})_{x}=\lambda_{i}(t) \mathsf w_{i},    
    \end{equation*}
    where $\lambda_i(t)=b(t)\lambda_i$.
     
    Fix $m\in\mathbb{N}^{*}$. Due to Caratheodory's theorem, there exist absolutely continuous functions $g_{im}=g_{im}(t)$ and $h_{im}=h_{im}(t)$ with $i\in\{1,2,...,m\}$ such that 
    \begin{equation*}
        \begin{split}
        	&t\in[0,T]\mapsto u_{m}(t)=\sum_{i=1}^{m}g_{im}(t) \mathsf w_{i} \in H_{a}^{1}(0,1) \\
          	&t\in[0,T]\mapsto v_{m}(t)=\sum_{i=1}^{m}h_{im}(t) \mathsf w_{i} \in H_{a}^{1}(0,1),
        \end{split}
    \end{equation*}
    satisfy, for any $w,\hat{w}\in [\mathsf w_{1}, \mathsf w_{2},...,\mathsf w_{m}]$  and $(\cdot,\cdot)=(\cdot,\cdot)_{L^{2}}$, that
    \begin{equation}
    	\label{eq:galerkin_system}
    	\left\{\begin{aligned}
    		&(u_{m,t},w) - b(t)((a(x) u_{m,x})_x,w) + (d_1 \sqrt{a}  u_{m,x},w) + (\tilde F_1(x,t,u_m, v_{m}),w) = 
            (H_1,w)  &&\text{in}&& Q, \\
    		&(v_{m,t},\hat{w}) - b(t)((a(x) v_{m,x})_x,\hat{w}) + (d_2 \sqrt{a} v_{m,x},\hat w) + (\tilde F_2(x,t,u_m,v_m),\hat w)  = (H_2,\hat{w})   &&\text{in}&& Q,\\
    		&u_{m}(0,t)=u_{m}(1,t)=0, \ v_{m}(0,t) = v_{m}(1,t) = 0 &&\text{on}&& (0,T), \\
    		& u_{m}(\cdot,0)\to u_{0} &&\text{in }&& (0,1), \\
    		& v_{m}(\cdot,0)\to v_{0} &&\text{in}&& (0,1).
    	\end{aligned}
    	\right.
    \end{equation}
    Estimate I: Taking $w=u_{m}$ and $\hat{w}=v_{m}$, then
    \begin{equation}\label{eq:d_dt_estI}
    	\begin{split}
    		&\frac{1}{2}\frac{d}{dt}\left(\|u_{m}\|^{2}+\|v_{m}\|^{2}\right) + b(t)\left(\|\sqrt{a}u_{m,x}\|^{2}+\|\sqrt{a}v_{m,x}\|^2\right) + (d_1 \sqrt{a} u_{m,x},u_m) + (d_2 \sqrt{a}v_{m,x},v_m)\\
    		&\quad +(\tilde F_1(x,t,u_m,v_m),u_m) + (\tilde F_2(x,t,u_m,v_m),v_m) = 
            (H_1,u_m) +(H_2,v_m).
    	\end{split}
    \end{equation}
    Then, by standard estimates, integration in time and using Gronwall's inequality we deduce
    \begin{equation}\label{energyestimate1}
    	\begin{split}
    		&\|u_m\|_{L^{\infty}(0,T,L^{2}(0,1))}^{2}+\|v_{m}\|_{L^{\infty}(0,T,L^{2}(0,1))}^{2}+\|\sqrt{a}u_{m,x}\|_{L^{2}(0,T,L^{2}(0,1))}^{2}+\|\sqrt{a}v_{m,x}\|_{L^{2}(0,T,L^{2}(0,1))}^{2}\\ 
    		&\leq e^{C_{*} T} \left( 
            \sum_{i=1}^{2}\|H_i\|^{2}_{L^{2}((0,T),L^2(0,1))}+\|u_{0}\|_{H^{1}_{a}(0,1)}+\|v_{0}\|^{2}_{H^{1}_a(0,1)}\right) =: \mathcal{K}_1.
    	\end{split}
    \end{equation}

    Estimate II: Taking $w=u_{m,t}$ and $\hat{w}=v_{m,t}$ in \eqref{eq:galerkin_system} and using estimate I, we get
    \begin{equation*}
    	\begin{split}
    		&\|u_{m,t}\|^2_{L^2(0,T,L^2(0,1))} +  \|v_{m,t}\|^2_{L^2(0,T,L^2(0,1))} + \|\sqrt{a}u_{m,x}\|^2_{L^\infty(0,T,L^2(0,1))} + \|\sqrt{a}v_{m,x}\|^2_{L^\infty(0,T,L^2(0,1))} \\
    		&\leq e^{DT} \left[ \|\sqrt{a}u_{m,x}(0)\|^2_{L^2(0,1)} + \|\sqrt{a}v_{m,x}(0)\|^2_{L^2(0,1)}
    		+ \sum_{i=1}^2  \|H_i\|^2_{L^2(0,T,L^2(0,1))} + 2\mathcal{K}_1 T  \right] =:\mathcal{K}_2.
    	\end{split}
    \end{equation*}

    Estimate III: Taking $w=-(a(x)u_{m,x})_x$ and $\hat{w}=-(a(x) v_{m,x})_x$ in \eqref{eq:galerkin_system}, we get that there exists a constant $D>0$ such that 
    \begin{equation}\label{eq:d_dt_estIII}
    	\begin{split}
    		&\frac{d}{dt}\left( \|\sqrt{a} u_{m,x} \|^2 + \|\sqrt{a}v_{m,x}\|^2\right) + \|(a u_{m,x})_x\|^2 + \|(a v_{m,x})_x\|^2 \\
    		&\leq D \left( 
            \sum_{i=1}^2  \|H_{i}\|^2 + \|u_m\|^2 + \|v_m\|^2
    		+ \|\sqrt{a} u_{m,x} \|^2 + \|\sqrt{a} v_{m,x}\|^2 
    		\right).
    	\end{split}
    \end{equation}
    Integrating in $t$ on $[0,t]$, using Grönwall's inequality,
    we have
    \begin{equation*}
    	\begin{split}
    		\|\sqrt{a} u_{m,x}\|^2_{L^\infty(0,T,L^2(0,1))}  + \|\sqrt{a}v_{m,x}\|^2_{L^\infty(0,T,L^2(0,1))} +  \|(a u_{m,x})_x\|^2_{L^2(0,T,L^2(0,1))} + \|(a v_{m,x})_x\|^2_{L^2(0,T,L^2(0,1))} \\
    		\leq e^{DT} \left[ \|\sqrt{a} u_{m,x}(0) \|^2 + \|\sqrt{a} v_{m,x}(0)\|^2 + 
            \sum_{i=1}^2  \|H_{i}\|^2_{L^2(0,T,L^2(0,1)} + 4\mathcal{K}_1 \right] =: \mathcal{K}_3.
    	\end{split}
    \end{equation*}

    Since $\mathcal{K}_1$, $\mathcal{K}_2$ and $\mathcal{K}_3$ do not depend on $m$, the three estimates above imply that the sequences $(u_m)$ and $(v_m)$ are bounded in 
    $L^2(0,T,L^2(0,1)) \cap H^1(0,T,H^2_a(0,1))$. 

    Therefore, there exist subsequences $(u_{m_j})$, $(v_{m_j})$ such that
    $$
        u_{m_j} \rightharpoonup u,\qquad v_{m_j} \rightharpoonup v, \qquad \text{as } j\to \infty, 
    $$
    weakly in $L^2(0,T,L^2(0,1)) \cap H^1(0,T,H^2_a(0,1))$.
    In fact, since our problem is weakly-degenerate, the immersion $H^2_a(0,1)$ in $H^1_a(0,1)$ is compact. By the theorem of Aubin-Lions we get
    $$
        u_{m_j} \to u,\qquad v_{m_j} \to v, \qquad \text{as } j\to \infty, 
    $$
    strongly in $L^2(0,T;H^1_a(0,1))$. Using the continuity of $F_i$, at least for a subsequence, we can pass to the limit in $m$ in all the terms of the approximate system \eqref{eq:galerkin_system}. The uniqueness is proved by standard methods for nonlinear systems.

    To obtain the exponential decay, we take the control $h=0$ and the external forces $H_i=0$. Combining \eqref{eq:d_dt_estI} and \eqref{eq:d_dt_estIII} 
    and using
    \eqref{cond_for_decay}, we have, for some $D>0,$
    \begin{equation*}
        \begin{split}
            &\frac{d}{dt}\left(\|u_{m}\|^{2}+\|v_{m}\|^{2} + \|\sqrt{a} u_{m,x} \|^2 + \|\sqrt{a}v_{m,x}\|^2\right) + D \left( \|\sqrt{a}u_{m,x}\|^{2}+\|\sqrt{a}v_{m,x}\|^{2} \right. \\
            &\left. \quad + \|(a u_{m,x})_x\|^2 + \|(a v_{m,x})_x\|^2 \right) 
            \leq 0
        \end{split}
    \end{equation*}
    From Poincaré inequality,
    $$
        \lambda_1(t) \|u\|^2 \leq \|\sqrt{a}u_x\|^2 \qquad\text{and}\qquad \lambda_1(t) \|\sqrt{a}u_x\|^2 \leq \| (a u_x)_x \|^2,
    $$
    where $\lambda_1(t)$ is the first eigenvalue of the operator $T(w) = -b(t)(a(x)w_{x})_{x}$ with Dirichlet boundary conditions. Since $0<b_0<b(t)$ then,
    $0<\tilde\lambda_1<\lambda_1(t)$ where $\tilde\lambda_1$ is the first eigenvalue of $-b_0 (a(x)w_{x})_{x}$. 
    Thus,
    \begin{equation*}
        \begin{split}
            &\frac{d}{dt}\left(\|u_{m}\|^{2}+\|v_{m}\|^{2} + \|\sqrt{a} u_{m,x} \|^2 + \|\sqrt{a}v_{m,x}\|^2\right) \\ 
            &
            + D \tilde\lambda_1 \left( \|u_{m}\|^{2}+\|v_{m}\|^{2} + \|\sqrt{a} u_{m,x} \|^2 + \|\sqrt{a} v_{m,x}\|^2 \right) \leq 0
        \end{split}
    \end{equation*}
    This implies that
    $$
        \left(\|u_{m}\|^{2}+\|v_{m}\|^{2} + \|\sqrt{a} u_{m,x} \|^2 + \|\sqrt{a}v_{m,x}\|^2\right) \leq C(u_0,v_0) 
        e^{-D\tilde\lambda_1t}.
    $$  
\end{proof}

\section{Study of the linearized system.}
\label{sec:linearized_system}
\quad	We will consider the following linearized system at zero:
\begin{equation}\label{eq:linearized_system1}
    \begin{cases}
		w_{1,t}-b(t)\left({a}(x)w_{1,x}\right)_x-B_1 (x,t)\sqrt{a}w_{1,x}+d_{11}(x,t)w_1\\
        \qquad + d_{12}(x,t)w_2=\overline{h} \cara_{_{{\omega}_0}} + G_1, & \ \ \ \text{in} \ \ \ {Q},\\
        w_{2,t}-b(t)\left({a}(x)w_{2,x}\right)_x-B_2 (x,t)\sqrt{a}w_{2,x}+d_{21}(x,t)w_1\\
        \qquad + d_{22}(x,t)w_2=G_2, & \ \ \ \text{in} \ \ \ {Q},\\
		w_1=0, \ \ w_2=0 & \ \ \ \text{on} \ \ \ {\Sigma},\\
		w_1(0)=y_0-\widetilde{y}_0=w_{1,0}, \ \ w_2(0)=z_0-\widetilde{z}_0=w_{2,0} & \ \ \ \text{in} \ \ \ (0,1).
    \end{cases}
\end{equation}
where 
\begin{align*}
    d_{11}(x,t):=\partial_3 F_1 (x,t,\widetilde{y},\widetilde{z}), & \quad d_{12}(x,t):=\partial_4 F_1 (x,t,\widetilde{y},\widetilde{z}),\\
    d_{21}(x,t):=\partial_3 F_2 (x,t,\widetilde{y},\widetilde{z}) &\,\,\, \text{and} \,\,\, d_{22}(x,t):=\partial_4 F_2 (x,t,\widetilde{y},\widetilde{z}).
\end{align*}

In the equation (\ref{eq:linearized_system1}) we define the new control $\overline{h} $ such that
\begin{equation}\label{new control}
    \overline{h} \cara_{_{{\omega}_0}}:={h}\cara_{_{{\omega}}}\widetilde{y}, \quad \text{where} \quad \omega_0\subset\subset \omega.
\end{equation}

{\textbf{Remark:}} It should be mentioned that the linearized problem (\ref{eq:linearized_system1}) allows us to prove results involving  {the} regularity of the control function $\overline{h}$ and  {also} additional estimates for the state $(w_1 ,w_2)$. Such results are important to obtain the null controllability of the nonlinear problem (\ref{nolineal}) with the new control function in the form:
$$
h\cara_{\omega}(w_1+\widetilde{y}):= \overline{h}\cara_{\omega_0}\left(\frac{\ 1 \ }{\widetilde{y}} w_1 +1 \right).
$$ 

As usual, the controllability of (\ref{eq:linearized_system1}) is closely related to the properties of the associated adjoint state. In this case, the adjoint of (\ref{eq:linearized_system1}) is given by

\begin{equation}\label{linealadjunto1}
	\begin{cases}
		-\varphi_{1,t}-b(t)\left({a}(x)\varphi_{1,x}\right)_x+\left(B(x,t)\sqrt{a}\varphi_1\right)_x+d_{11}(x,t)\varphi_1\\
        \qquad +d_{21}(x,t)\varphi_2=H_1, & \ \ \ \text{in} \ \ \ {Q},\\
        -\varphi_{2,t}-b(t)\left({a}(x)\varphi_{2,x}\right)_x+\left(B(x,t)\sqrt{a}\varphi_2\right)_x+d_{12}(x,t)\varphi_1\\
        \qquad +d_{22}(x,t)\varphi_2=H_2, & \ \ \ \text{in} \ \ \ {Q},\\
		\varphi_1=0, \ \ \varphi_2=0 & \ \ \ \text{on} \ \ \ {\Sigma},\\
		\varphi_1(T)=\varphi_1^T, \ \ \varphi_2(T)=\varphi_2^T, & \ \ \ \text{in} \ \ \ (0,1),
	\end{cases}
\end{equation}
where $\varphi_{i}^T \in L^2(0,1)$ and $H_i \in L^2(Q)$ for $i=1,2$.

Next, we define several weight functions which will be useful in the sequel. The basic weight will be a function $\Psi : [0,1] \to \R$ of class $C^2$ such that
$$
    \Psi(x) = 
    \begin{cases}
        \int_0^x \frac{s}{a(s)} ds, \quad x \in [0,\alpha'), \\
        -\int_{\beta'}^x \frac{s}{a(s)} ds, \quad x \in [\beta',1],
    \end{cases}
$$
with $\omega'=(\alpha',\beta') \subset\subset \omega$.

Let us introduce the function $\sigma =\sigma (t)$ satisfying
$$
\left\{
\begin{array}{ll}
\sigma \in C^{\infty}([0,T])\,,\\
\sigma (t)\geq \frac{T^2}{4}\,, & t \in [0,T/2]\,, \\
\sigma (t)=t(T-t)\,, & t \in [T/2,T], \,\\
\sigma(0) >0\,.
\end{array}
\right.
$$
and define
    $$
        \tau(t) = \frac{1}{\sigma(t)}, \quad \zeta(x,t) = \tau(t) (e^{\lambda(|\Psi|_\infty + \Psi)}) \quad \text{and} \quad A(x,t)  = \tau(t) (e^{\lambda(|\Psi|_\infty + \Psi)}-e^{3\lambda|\Psi|_\infty}).
    $$
    
The following Carleman estimate was proved in \cite{GYL-Carleman-2025}
\begin{propo} 
		\label{prop:carleman}
		There exist positive constants $C$, $\lambda_0$ and $s_0$ such that, for any $s \geq s_0$, $\lambda \geq \lambda_0$ and any $\varphi_1^T\in L^2(Q)$, the corresponding solution $(\varphi_1,\varphi_2)$ of (\ref{linealadjunto1}) satisfies
		\begin{equation*}
		    \begin{split}
		        &\int_Q e^{2s A}((s\lambda) \zeta b^2 a (|\varphi_{1,x}|^2 + |\varphi_{2,x}|^2) + (s\lambda)^2 {\zeta}^2 b^2 (|\varphi_1|^2 + |\varphi_2|^2))dxdt \\
                &\qquad \leq C \left( \int_Q e^{2s A} s^4 \lambda^4 \zeta^4 (|H_1|^2+|H_2|^2)dxdt + \int_{\omega_0 \times (0,T)} e^{2s A} s^8 \lambda^8 \zeta^8 |\varphi_1|^2 dxdt\right).
		    \end{split}
		\end{equation*}
	\end{propo}

As a corollary, we get the following observability inequality.
\begin{coro} 
    \label{cor:observability}
    There exist positive constants $C$, $\lambda_0$ and $s_0$ such that, for any $s \geq s_0$, $\lambda \geq \lambda_0$ and any $\varphi_1^T\in L^2(Q)$, the corresponding solution $(\varphi_1,\varphi_2)$ of (\ref{linealadjunto1}) with $H_1 = H_2 = 0$, satisfies
    \begin{equation}\label{Observability1}
        \|\varphi_1(0)\|^2_{L^2(0,1)} + \|\varphi_2(0)\|^2_{L^2(0,1)} \leq C \int_{\omega_0 \times (0,T)} e^{2s A} s^8 \lambda^8 \zeta^8 |\varphi_1|^2dxdt.
    \end{equation}	
\end{coro}

In the sequel, let us set weights that depend only on $t$, in the following sense:
\begin{equation*}
    \begin{cases}
        A^*(t) = \displaystyle\max_{x \in (0,1)} A(x,t), \qquad \hat A(t) = \displaystyle\min_{x \in (0,1)} A(x,t), \\
        \zeta^*(t) = \displaystyle\max_{x \in (0,1)} \zeta(x,t), \qquad \hat \zeta(t) = \displaystyle\min_{x \in (0,1)} \zeta(x,t).
    \end{cases}
\end{equation*}
Observe that $A^*(t) < 0$, $\hat A(t) < 0$ and that $\zeta^*(t) / \hat \zeta(t)$ does not depend on $t$ and is equal to some constant $\zeta_0 \in \R$. 

Thus, Proposition \ref{prop:carleman} and Corollary \ref{cor:observability}  imply the following corollary where the weights depend only on $t$.

\begin{coro} 
		\label{cor:carleman_pesos_t}
		There exist positive constants $C$, $\lambda_0$ and $s_0$ such that, for any $s \geq s_0$, $\lambda \geq \lambda_0$ and any $\varphi_1^T \in L^2(Q)$, the corresponding solution $(\varphi_1,\varphi_2)$ of (\ref{linealadjunto1}) satisfies
		\begin{equation}
			\label{pesos_t_Carleman for eq:adjoint_optimality_system}
            \begin{split}
                &\|\varphi_1(0)\|^2_{L^2(0,1)} + 
                \int_Q e^{2s \hat A} \left[(s\lambda) \hat \zeta b^2 a (|\varphi_{1,x}|^2 + |\varphi_{2,x}|^2) + (s\lambda)^2 {\hat \zeta}^2 b^2 (|\varphi_1|^2+|\varphi_2|^2) \right] dx dt \\
                &\qquad \leq C \left( \int_Q e^{2s A^*} (\zeta^*)^4 (|H_1|^2+|H_2|^2)dxdt + \int_{\omega \times (0,T)} e^{2s A^*} (\zeta^*)^8 |\varphi_1|^2 dxdt \right).    
            \end{split}
		\end{equation}
	\end{coro}

\subsection{A null controllability result for the linear system}
Our aim is to prove a global null controllability result for the linear system (\ref{eq:linearized_system1}). To this end, we need to define the weight functions:
$$
\rho_j :=e^{-sA^*}  (\zeta^*)^{-j} \quad \text{with the norm} \quad \Vert \rho_j w\Vert_{L^2 (Q)}:=\Bigl(\int_Q \rho_{j}^2\vert w\vert^2 \Bigr)^{1/2},\,\,j \in \mathbb{N} \cup \{ 0\}
$$
which satisfy
\begin{equation}\label{eq:compara_rhos}
     \rho_{5} \leq C\rho_{4}\leq C \rho_3 \leq C\rho_{2}\leq C \rho_{1}, \quad \text{and} \quad\rho_{3}^{2}=\rho_{4}\rho_{2} .   
\end{equation}

The next result establishes the null controllability of system (\ref{eq:linearized_system1}), which is a very well known result. The estimates in this result are also known in the literature and have been used with various purposes, although we believe that is useful to briefly recall them here, as well as their proofs.

\begin{teo}
    \label{prop:linear_control}
		If $(w_{1}^0 ,w_{2}^0)\in [L^{2}(0,1)]^2$, $\rho_2 G_i \in L^2(Q),\,\, (i=1,2)$, then there exists a control $\overline{h}\in L^{2}(\omega \times (0,T))$ with associated states $w_1, w_2 \in C^{0}([0,T];L^{2}(0,1))\cap L^{2}(0,T;H^{1}_{a})$, solutions of system \eqref{eq:linearized_system1}, such that
		\begin{equation}\label{estimate for solution}
			\int_Q \rho_2^2 |w_1|^2 dxdt + \int_Q \rho_2^2 |w_2|^2 dxdt + \int_{\omega \times (0,T)} \rho_4^2 |\overline{h}|^2dxdt \leq
			C \kappa_{0}(G_{1},G_{2},w_{1,0},w_{2,0}),   
		\end{equation}
		where 
                $$\kappa_{0}(G_{1},G_{2},w_{1,0},w_{2,0})= \Vert \rho_2 G_1 \Vert^2_{L^2(Q)} + \Vert\rho_2 G_2 \Vert^2_{L^2(Q)} + \Vert w_{1,0}\Vert^2_{L_2(0,1)} + \Vert w_{2,0} \Vert^2_{L_2(0,1)}\,.$$
        
        In particular, $w_1(x,T)=0$ and $w_2(x,T)=0$, for all $x\in [0,1]$.
\end{teo}

\begin{proof}
    We define $\Lambda_0 =\{(\phi_1 ,\phi_2) \in [C^2(\overline{Q})]^3): \phi_1 =\phi_2 =0 \ \  \text{on} \  \ \Sigma\}$ and denote
    $$L(w_1 ,w_2)=(L_1(w_1,w_2),L_2(w_1,w_2)) \quad \text{and} \quad L^{\ast}(\phi_1 ,\phi_2)=(L_1^{\ast}(\phi_1 ,\phi_2),L_2^{\ast}(\phi_1 ,\phi_2))$$
    with
\begin{align*}
L_1(w_1,w_2)  & =w_{1,t}-b(t)\left({a}(x)w_{1,x}\right)_x-B(x,t)\sqrt{a}w_{1,x}+d_{11}(x,t)w_1 +d_{12}(x,t)w_2, \\
L_2(w_1,w_2)  & =  w_{2,t}-b(t)\left({a}(x)w_{2,x}\right)_x-B(x,t)\sqrt{a}w_{2,x}+d_{21}(x,t)w_1 +d_{22}(x,t) w_2, \\
L_1^{\ast}(\phi_1 ,\phi_2) & =-\varphi_{1,t}-b(t)\left({a}(x)\phi_{1,x}\right)_x+\left(B(x,t)\sqrt{a}\phi_1\right)_x+d_{11}(x,t)\phi_1 +d_{21}(x,t)\phi_2 ,\\
L_2^{\ast}(\phi_1 ,\phi_2)  & =-\phi_{2,t}-b(t)\left({a}(x)\phi_{2,x}\right)_x+\left(B(x,t)\sqrt{a}\phi_2\right)_x+d_{12}(x,t)\phi_1 +d_{22}(x,t) \phi_2 .
\end{align*}

We consider on $\Lambda_0$ the continuous bilinear form $\beta: \Lambda_0 \times \Lambda_0 \rightarrow \mathbb{R}$ defined by
$$
\beta \bigl((\phi_1 ,\phi_2),(\widetilde{\phi}_1,\widetilde{\phi}_2)\bigr):=\int_{Q}\bigl[\rho_2^{-2} L^{\ast}(\phi_1 ,\phi_2)\cdot L^{\ast}(\widetilde{\phi}_1,\widetilde{\phi}_2) +\cara_{\omega_0} \rho_4^{-2} \bigl((\phi_1 ,\phi_2)\cdot(\widetilde{\phi}_1,\widetilde{\phi}_2)\bigr)\bigr] dxdt.
$$

Let us denote $\Lambda :=\overline{\Lambda_0}^{\beta (\cdot ,\cdot)}$, where $\Lambda$ is the completion of $\Lambda_0$ under the norm induced by the inner product $\beta(\cdot ,\cdot)$. On the other hand, define $$\ell(\phi_1 ,\phi_2)=\int_{0}^{1} (w_{1,0} \cdot \phi_1 (x,0)+ w_{2,0} \cdot \phi_2 (x,0)) \,dx + \int_{Q} (H_1(x,t)\cdot \phi_1 + H_2(x,t) \cdot \phi_2)\, dxdt .$$

From Corollary \ref{cor:carleman_pesos_t}, it is proven that $\ell$ is bounded, and since $\beta$ is coercive in $\Lambda$ with the inner product generated by $\beta$, then from Lax-Milgram Theorem, it can be verified that there exists a unique solution $\phi=(\phi_1 ,\phi_2) \in \Lambda$ solution of
\begin{equation}\label{l-m}
    \beta(\phi ,w)=\ell (w); \,\,\, \forall w=(w_1 ,w_2) \in \Lambda\,.
\end{equation}

Define 
\begin{equation}
    \label{est-cont}
w_1 =-\rho_{2}^{-2} L_{1}^{\ast} \phi, \quad w_2 =-\rho_{2}^{-2} L_{2}^{\ast} \phi \quad \text{and} \quad \overline{h}=-\rho_{4}^{-2} \phi_1 \vert_{\omega_0 \times (0,T)}.
\end{equation}
From (\ref{l-m}), we obtain (\ref{estimate for solution}).
\end{proof}

The next result, stated below, is a immediate consequence of Theorem \ref{prop:linear_control}. Arguing in a similar way, we will establish the regularity of the control $\overline{h}$, which will be very useful to work out estimates involving multiplicative control in Section 4. The aforementioned regularity is important to control the problem (\ref{nolineal}).

\begin{coro}\label{rcont}
Let $(w_{1}^0 ,w_{2}^0)\in [L^{2}(0,1)]^2$, $\rho_2 G_i \in L^2(Q),\,\, (i=1,2)$. From the definition of the weight $\rho_j , j \in \mathbb{N} \cup \{ 0\}$, the following holds
\begin{equation}
		\label{cr1}
		\rho_{5} \overline{h} \in L^{2}(0,T;H^{1}_{a}(0,1) \cap H^{2}(0,1)) \cap L^{\infty}(0,T;H^{1}_{a}(0,1)) \qquad and
	\end{equation}
	\vspace{-0.5cm}
	\begin{equation}
		\label{cr2}
		\Vert \rho_{5} \overline{h}\|_{L^{2}(0,T;H^{1}_{a}(0,1)}+\Vert \rho_{5} \overline{h}\|_{L^{\infty}(0,T;H^{1}_{a}(0,1))}\leq C \kappa_{0}(G_{1},G_{2},w_{1,0},w_{2,0}),
	\end{equation}
	where $$\kappa_{0}(G_{1},G_{2},w_{1,0},w_{2,0})= \Vert \rho_2 G_1 \Vert^2_{L^2(Q)} + \Vert\rho_2 G_2 \Vert^2_{L^2(Q)} + \Vert w_{1,0}\Vert^2_{L_2(0,1)} + \Vert w_{2,0} \Vert^2_{L_2(0,1)}\,.$$
\end{coro}
\begin{proof}
Proceeding as in the proof of Theorem \ref{prop:linear_control} to $(\overline{\phi}_1,\overline{\phi}_2)=(\rho_{3}^{-1}\phi_1 ,\rho_{3}^{-1}\phi_2)$ in (\ref{est-cont}), we have that 
\begin{equation}\label{est}
   L_{1}^{\ast} (\overline{\phi}_1,\overline{\phi}_2) =-\rho_{3}^{-1}\rho_{2}^{2}w_1 -(\rho_{3}^{-1})_t \phi_1, \quad L_{2}^{\ast} (\overline{\phi}_1,\overline{\phi}_2) =-\rho_{3}^{-1}\rho_{2}^{2}w_2 -(\rho_{3}^{-1})_t \phi_2 , \end{equation}
   where
   \begin{align*}
L_1^{\ast}(\overline{\phi}_1,\overline{\phi}_2) & =-(\rho_{3}^{-1}\phi_1)_{t}-b(t)\left({a}(x)(\rho_{3}^{-1}\phi_1))_{x}\right)_x+\left(B(x,t)\sqrt{a}(\rho_{3}^{-1}\phi_1)\right)_x+d_{11}(x,t)(\rho_{3}^{-1}\phi_1)\\ & \quad +d_{21}(x,t)(\rho_{3}^{-1}\phi_2) ,\\
L_2^{\ast}(\overline{\phi}_1,\overline{\phi}_2)  & =-(\rho_{3}^{-1}\phi_2)_{t}-b(t)\left({a}(x)(\rho_{3}^{-1}\phi_2)_{x}\right)_x+\left(B(x,t)\sqrt{a}(\rho_{3}^{-1}\phi_2)\right)_x\\
&\quad +d_{12}(x,t)(\rho_{3}^{-1}\phi_1) +d_{22}(x,t) (\rho_{3}^{-1}\phi_2) .
\end{align*}

   and
   \begin{equation}\label{cont}
   \rho_{3}^{-1}\overline{h}=-\rho_{3}^{-1}\rho_{4}^{-2} \phi_1 \vert_{\omega_0 \times (0,T)}. 
\end{equation}

It is not difficult to check that $$\vert \rho_{3}^{-1}\rho_{2}^{2}w_1\vert \leq C \vert \rho_2 w_1\vert, \quad\vert \rho_{3}^{-1}\rho_{2}^{2}w_2\vert \leq C \vert \rho_2 w_2\vert$$ and $$\vert (\rho_{3}^{-1})_t \phi_1\vert \leq C\vert \rho_{2}^{-1} \phi_1\vert , \quad \vert (\rho_{3}^{-1})_t \phi_2\vert \leq C\vert \rho_{2}^{-1} \phi_2 \vert.$$

So, from Theorem \ref{prop:linear_control}, one has 
$$\Vert \rho_{3}^{-1}\rho_{2}^{2}w_1 \Vert_{L^2 (Q)}^2 +\Vert \rho_{3}^{-1}\rho_{2}^{2}w_2 \Vert_{L^2 (Q)}^2 \leq C\left( \Vert w_{1,0}\Vert_{L^2 (0,1)}^{2}+\Vert w_{2,0}\Vert_{L^2 (0,1)}^{2}+\sum_{i=1}^{2}\Vert \rho_{2} G_i \Vert_{L^2(Q)}^{2} \right)$$
and from Carleman estimate, we can deduce
$$\Vert (\rho_{3}^{-1})_t \phi_1 \Vert_{L^2 (Q)}^2 +\Vert (\rho_{3}^{-1})_t \phi_2 \Vert_{L^2 (Q)}^2\leq C\left( \Vert w_{1,0}\Vert_{L^2 (0,1)}^{2}+\Vert w_{2,0}\Vert_{L^2 (0,1)}^{2}+\sum_{i=1}^{2}\Vert \rho_{2} G_i \Vert_{L^2(Q)}^{2}\right).$$
Therefore, 
$$(\rho_{3}^{-1}\rho_{2}^{2}w_1, \rho_{3}^{-1}\rho_{2}^{2}w_2, (\rho_{3}^{-1})_t \phi_1, (\rho_{3}^{-1})_t \phi_2) \in L^2 (Q) \quad \text{and}$$
\begin{equation}\label{lag.2}
\Vert \rho_{3}^{-1}\rho_{2}^{2}w_1 \Vert_{L^2 (Q)}^2 +\Vert \rho_{3}^{-1}\rho_{2}^{2}w_2 \Vert_{L^2 (Q)}^2 +\Vert (\rho_{3}^{-1})_t \phi_1 \Vert_{L^2 (Q)}^2 +\Vert (\rho_{3}^{-1})_t \phi_2 \Vert_{L^2 (Q)}^2 \leq C\kappa_{0}(G_{1},G_{2},w_{1,0},w_{2,0})\,.
 \end{equation}

 In view of (\ref{est}), (\ref{lag.2}) and parabolic regularity, we have that
 $$(\overline{\phi}_1,\overline{\phi}_2)=(\rho_{3}^{-1}\phi_1 ,\rho_{3}^{-1}\phi_2) \in L^{2}(0,T;H^{1}_{a}(0,1) \cap H^{2}(0,1)) \cap L^{\infty}(0,T;H^{1}_{a}(0,1))\,.$$

 Now, from (\ref{cont}) one has
 \begin{equation}\label{rc}
 \begin{split}
    \rho_{3}^{-1}\phi_1 =& - \rho_{3}^{-1}\rho_{4}^{2} \overline{h}   \\
        =& -(e^{sA^*}  (\zeta^*)^{3})\cdot (e^{-2sA^*}  (\zeta^*)^{-8}) \overline{h}\\
        =& (e^{-sA^*}  (\zeta^*)^{-5}) \overline{h}\\
    =&  -\rho_5 \overline{h} .  
 \end{split}
\end{equation}

 Therefore,
 $$\rho_5 \overline{h} \in L^{2}(0,T;H^{1}_{a}(0,1) \cap H^{2}(0,1)) \cap L^{\infty}(0,T;H^{1}_{a}(0,1))$$

 and
 $$\Vert \rho_5 \overline{h} \Vert_{L^2(0,T; H^1_a (0,1) \cap H^2 (0,1))} + \Vert \rho_5 \overline{h}  \Vert_{L^\infty(0,T; H^1_a(0,1))} \leq C\kappa_{0}(G_{1},G_{2},w_{1,0},w_{2,0}). $$
\end{proof}

\subsection{Weighted energy estimates}

In order to get the local null controllability of the nonlinear system (\ref{nolineal}) we need the following additional estimates.

\begin{propo}
\label{addicional_estimates_case_linear}
Under the hypothesis of Theorem \ref{prop:linear_control}, we have, furthermore, that the control $\overline h\in L^{2}(\omega_0 \times (0,T))$ and the associated states $(w_1 ,w_2) \in C^{0}([0,T];L^{2}(0,1))\cap L^{2}(0,T;H^{1}_{a}(0,1))$, solution of (\ref{eq:linearized_system1}), satisfies the additional estimates 
\begin{equation}\label{des Proposition 5}
        \begin{array}{c}
            \displaystyle \sup_{t \in [0,T]}(\rho_3^{2}\|w_1\|^{2}_{L^{2}(0,1)}) + \displaystyle\sup_{t \in [0,T]}(\rho_3^{2}\|w_2\|^{2}_{L^{2}(0,1)})\vspace{0.1cm}\\
            +\displaystyle\int_{Q}\rho_3^{2} a(x)(| |w_{1,x}|^{2}+|w_{2,x}|^{2})dxdt\ \leq C \kappa_{0}(G_{1},G_{2},w_{1,0},w_{2,0})  
        \end{array}
    \end{equation}
    and, if $w_{1,0}, \ w_{2,0} \in H^{1}_{a}(0,1)$ 
    \begin{equation}\label{des Proposition 6}
        \begin{array}{c}
            \displaystyle\sup_{t \in [0,T]}(\rho_{4}^{2}\|\sqrt{a}w_{1,x} \|^{2}_{L^{2}(0,1)}) + \displaystyle\sup_{t \in [0,T]}(\rho_{4}^{2}\|\sqrt{a}w_{2,x} \|^{2}_{L^{2}(0,1)})\\
            + \displaystyle\int_{Q}\rho_{4}^{2}(|w_{1,t}|^{2}+|w_{2,t}|^{2}+|(a(x)w_{1,x})_{x}|^{2} + |(a(x)w_{2,x})_{x}|^{2} )dxdt\\
            \leq C \kappa_{1}(G_{1},G_{2},w_{1,0},w_{2,0}),  
        \end{array}
    \end{equation}
    where $\kappa_{1}(G_{1},G_{1},w_{1,0},w_{2,0})= \|\rho_2 G_1 \|^2_{L^2(Q)} + \|\rho_2 G_2 \|^2_{L^2(Q)} + \|w_{1,0}\|^2_{H^{1}_{a}} + \|w_{2,0}\|^2_{H^{1}_{a}}$. 
\end{propo}
\begin{proof}
We proceed following the steps of \cite{DemarqueLimacoViana_deg_eq2018}, see also \cite{Joao_Juan_Suerlan-25,GLoY-2026}.
    	
	Let us start multiplying the first equation in (\ref{eq:linearized_system1}) by $\rho_2^{2} w_1$ and the second one by $\rho_3^{2} w_2$ and let us integrate over $[0,1]$. Hence, using that $\rho_3^{2} = \rho_{2}\rho_{4}$, and $\rho_{4}\leq C\rho_2$, we have
	\begin{equation}\label{second estimate3}
		\begin{array}{l}
			\dfrac{1}{2}\dfrac{d}{dt}\displaystyle\int_{0}^{1}\rho_3^{2}(|w_1|^{2}+|w_2|^{2})dx + b(t) \displaystyle\int_{0}^{1}\rho_3^{2}a(x)(|w_{1,x}|^{2}+|w_{2,x}|^{2}) dx \\
            + \displaystyle\int_0^1 \rho_3^{2} \sqrt{a(x)}(B_1 w_{1,x} w_1+B_2 w_{2,x} w_2) dx   \\
			\leq C\left(\displaystyle\int_{0}^{1}\rho_{2}^{2}(|w_1|^{2}+|w_2|^{2})dx +  \displaystyle\int_{O}\rho_{4}^{2}|h|^{2}dx +
			\displaystyle\int_{0}^{1}\rho_{2}^{2}(|G_{1}|^{2}+|G_{2}|^{2})dx\right)\hspace{0.1cm}+ \mathcal{M},
		\end{array}
	\end{equation}
	where $\mathcal{M} = \displaystyle\int_{0}^{1}\rho_3(\rho_3)_{t}(|w_1|^{2}+|w_2|^{2})dx$ 
	and $(\cdot)_{t}=\frac{d}{dt}(\cdot)$.
    Recall that
	$A^*(t) = C_1 \tau(t)$, and $\zeta^*(t) = C_2 \tau(t)$, then we have that $(A^*)_t=\bar C (\zeta^*)_{t}$ and consequently
	\begin{equation*}
		\begin{split}
			\rho_3(\rho_3)_{t} &= e^{-sA^*}(\zeta^*)^{-3}\left(-se^{-sA^*} A^*_{t}(\zeta^*)^{-3} -3 e^{-sA^*}(\zeta^*)^{-4}(\zeta^*)_{t}\right)    \\
			&= -e^{-2sA^*}(\zeta^*)^{-4}(\zeta^*)_{t}\left(s(\zeta^*)^{-2}\bar{C} + 3(\zeta^*)^{-3} \right)  \\
			&=-\rho_{2}^{2}(\zeta^*)_{t}\left(s(\zeta^*)^{-2}\bar{C} + 3(\zeta^*)^{-3} \right).
		\end{split}
	\end{equation*}
	Thus, for any $t\in [0,T)$,
	\begin{equation*}
		\begin{array}{l}
			|\rho_3(\rho_3)_{t}|\leq C\rho_{2}^{2}\tau^{2}|s(\zeta^*)^{-2}\bar{C} + 3(\zeta^*)^{-3}|  
			\leq C\rho_{2}^{2}|s\bar{C}+3(\zeta^*)^{-1}| \leq C\rho_{2}^{2},
		\end{array}
	\end{equation*}
	and then, we obtain
	\begin{equation*}
		\mathcal{M}\leq C \displaystyle\int_{0}^{1}{\rho_{2}^{2}}(|w_1|^{2}+|w_2|^{2})dx.
	\end{equation*}
	
	\noindent Therefore, using Young's inequality, (\ref{second estimate3}) becomes, for small $\epsilon>0$ and using the boundedness of $c_{ij}$ and $d_i$, for $i,j=1,2$, 
	\begin{multline*}
		\dfrac{1}{2}\dfrac{d}{dt}\displaystyle\int_{0}^{1}\rho_3^{2}(|w_1|^{2}+|w_2|^{2})dx + b(t) \displaystyle\int_{0}^{1}\rho_3^{2} a(x)(|w_{1,x}|^{2}+|w_{2,x}|^{2})dx\\
		\leq C\left(\displaystyle\int_{0}^{1}\rho_{2}^{2}(|w_1|^{2} +|w_2|^{2} )dx + \displaystyle\int_{\omega}\rho_{4}^{2}|h|^{2}dx + \displaystyle\int_{0}^{1}\rho_{2}^{2}(|G_{1}|^{2}+|G_{2}|^{2})dx\right)\\
		+ \epsilon \displaystyle\int_{0}^{1}\rho_3^{2} a(x)(|w_{1,x}|^{2}+|w_{2,x}|^{2})dx + C_\epsilon \int_0^1 \left( |w_1|^{2} +|w_2|^{2} \right)dx.
	\end{multline*}
    
	Thus, since $b(t)$ is bounded and $\rho_0$ is bounded by below, there is constant $D>0$ such that
	\begin{multline*}
		\dfrac{1}{2}\dfrac{d}{dt}\displaystyle\int_{0}^{1}\rho_3^{2}(|y_1|^{2}+|y_2|^{2})dx + \displaystyle\int_{0}^{1}\rho_3^{2} a(x)(|w_{1,x}|^{2}+|w_{2,x}|^{2})dx\\
		\leq D\left(\displaystyle\int_{0}^{1}\rho_{2}^{2}(|w_1|^{2} +|w_2|^{2})dx + \displaystyle\int_{\omega}\rho_{4}^{2}|h|^{2}dx + \displaystyle\int_{0}^{1}\rho_{2}^{2}(|G_{1}|^{2}+|G_{2}|^{2})dx\right),
	\end{multline*}
	and, integrating in time,  we conclude (\ref{des Proposition 5}) by the Linear Theorem.


    Now, to prove (\ref{des Proposition 6}), multiply the  first equation in (\ref{eq:linearized_system1}) by $\rho_{4}^{2} w_{1,t}$ and the second one by $\rho_{4}^{2} w_{2,t}$ and integrate over $[0,1]$. We get
    \begin{equation*}
        \begin{array}{l}
            \displaystyle\int_{0}^{1}\rho_{4}^{2}(|w_{1,t}|^{2}+|w_{2,t}|^{2})dx + \dfrac{1}{2} b(t) \displaystyle\int_{0}^{1}\rho_{4}^{2} a(x) \dfrac{d}{dt}(|w_{1,x}|^{2}+|w_{2,x}|^{2})dx  \vspace{0.1cm}\\
             \quad+ \displaystyle\int_0^1 \rho_{4}^{2} (d_{11} w_1 + d_{12}w_2) w_{1,t} dx +\int_0^1 \rho_{4}^{2} (d_{21}w_1 + d_{22}w_2) w_{2,t} dx
              \\
             \quad - \displaystyle \int_0^1 \rho_{4}^{2} \sqrt{a}w_{1,x} w_{1,t} dx -\int_0^1 \rho_{4}^{2} \sqrt{a}w_{2,x} w_{2,t} dx \\
            \leq C\left(\displaystyle\int_{\omega}\rho_{4}^{2}|h|^{2}dx + \displaystyle\int_{0}^{1}\rho_{2}^{2}(|G_1|^{2}+|G_2|^{2})dx\right) + \dfrac{1}{4}\displaystyle\int_{0}^{1}\rho_{4}^{2}(|w_{1,t}|^{2}+|w_{2,t}|^{2})dx.
        \end{array}
    \end{equation*}
    Thus, using that $\rho_3^{2} = \rho_2 \rho_4$, Young's inequality and the boundedness of $d_{ij}$ and $B_i$, for $i,j=1,2$, we get
    \begin{equation}\label{third estimate}
        \begin{array}{l}
            \displaystyle\int_{0}^{1}\rho_{4}^{2}(|w_{1,t}|^{2}+|w_{2,t}|^{2})dx + \dfrac{1}{2} b(t) \dfrac{d}{dt}\displaystyle\int_{0}^{1}\rho_{4}^{2} a(x)(|w_{1,x}|^{2}+|w_{2,x}|^{2})dx  \vspace{0.1cm}\\
            \leq D\left(\displaystyle\int_{0}^{1}\rho_{2}^{2}(|w_1|^{2} + |w_2|^{2})dx + \displaystyle\int_{O}\rho_{4}^{2}|h|^{2}dx + \displaystyle\int_{0}^{1}\rho_{2}^{2}(|G_1|^{2}+|G_2|^{2})dx\right) \vspace{0.1cm}\\
            \quad + \dfrac{1}{2}\displaystyle\int_{0}^{1}\rho_{4}^{2}(|w_{1,t}|^{2}+|w_{2,t}|^{2})dx + |\widetilde{\mathcal{M}}|,
        \end{array}
    \end{equation}
    where $\widetilde{\mathcal{M}}= \dfrac{1}{2}\displaystyle\int_{0}^{1}(\rho_{4}^{2})_{t}\, a(x)(|w_{1,x}|^{2}+|w_{2,x}|^{2})dx$.
    Since $|\zeta_{t}|\leq C\zeta^{2}$,  
    we have that $|(\rho^{2}_{1})_{t}|\leq C\rho_3^{2}$. 
    Hence,
    $$
        |\widetilde{\mathcal{M}}| \leq C\displaystyle \int_{0}^{1}\rho_3^{2}\, a(x)(|w_{1,x}|^{2}+|w_{2,x}|^{2})dx.
    $$ 
    
    Thus, (\ref{third estimate}) gives 
    \begin{equation*}
        \begin{array}{l}
            \dfrac{1}{2}\displaystyle\int_{0}^{1}\rho_{4}^{2}(|w_{1,t}|^{2}+|w_{2,t}|^{2})dx + \dfrac{1}{2} b(t)  \dfrac{d}{dt}\displaystyle\int_{0}^{1}\rho_{4}^{2} a(x)(|w_{1,x}|^{2}+|w_{2,x}|^{2})dx  \vspace{0.1cm}\\
             \leq D\left(\displaystyle\int_{0}^{1}\rho_{2}^{2}(|w_1|^{2} + |w_2|^{2})dx + \displaystyle\int_{\omega}\rho_{4}^{2}|h|^{2}dx + \displaystyle\int_{0}^{1}\rho_{2}^{2}(|G_1|^{2}+|G_2|^{2})dx\right) \vspace{0.1cm}\\
            \, + \, C\displaystyle\int_{0}^{1}\rho_3^{2}\, a(x)(|w_{1,x}|^{2}+|w_{2,x}|^{2})dx.
        \end{array}
    \end{equation*}
    
     Integrating the previous inequality from $0$ to $t$ and using the boundedness of $b(t)$ and the first estimate, (\ref{des Proposition 5}), we arrive to
    \begin{equation}\label{primeira da des proposição 6}
        \begin{array}{c}
            \displaystyle\sup_{t \in [0,T]}(\rho_{4}^{2}\|\sqrt{a}w_{1,x} \|^{2}_{L^{2}(0,1)})+ \displaystyle\sup_{t \in [0,T]}(\rho_{4}^{2}\|\sqrt{a}w_{2,x} \|^{2}_{L^{2}(0,1)}) + \displaystyle\int_{Q}\rho_{4}^{2}(|w_{1,t}|^{2}+|w_{2,t}|^{2})dxdt\\
            \leq C \kappa_{1}(G_1,G_2,w_{1,0},w_{2,0}).
         \end{array}
     \end{equation}
\end{proof}

\section{Null controllability of the nonlinear system}
\label{sec:control for nonlinear system}

\qquad In the present section, we will prove the local and global null controllability of the nonlinear systems (\ref{nolineal}). The null controllability of this system is equivalent to the local exact controllability of the solution to (\ref{eq:PDE}). To this end, we will apply Liusternik's Inverse Function Theorem (\cite{Alekseev}) in infinite dimensional spaces.

Initially, remember that from (\ref{new control})
$$
\overline{h} \cara_{_{{\omega}_0}}:={h}\cara_{_{{\omega}}}\widetilde{y} \quad \text{with} \quad \vert \widetilde{y} \vert \geq C >0\quad \text{and} \quad \omega_0\subset\subset \omega.
$$
In this way, we see that
$$
h\cara_{\omega}(w_1+\widetilde{y})= \overline{h}\cara_{\omega_0}\left(\frac{\ 1 \ }{\widetilde{y}} w_1 +1 \right).
$$ 

Thereby, we will rewrite the nonlinear system (\ref{nolineal}) in the following form:
\begin{equation}\label{nolineal1}
    \begin{cases}
		w_{1,t}-b(t)\left({a}(x)w_{1,x}\right)_x-B_1(x,t)\sqrt{a}w_{1,x}+F_1\left(x,t,w_1+\widetilde{y},w_2+\widetilde{z}\right) \\
        \qquad -F_1\left(x,t,\widetilde{y},\widetilde{z}\right)=\overline{h}\cara_{\omega_0}\left(\frac{\ 1 \ }{\widetilde{y}} w_1 +1 \right), & \ \ \ \text{in} \ \ \ {Q},\\
        w_{2,t}-b(t)\left({a}(x)w_{2,x}\right)_x-B_2(x,t)\sqrt{a}w_{2,x}+F_2\left(x,t,w_1+\widetilde{y},w_2+\widetilde{z}\right) \\
        \qquad - F_2\left(x,t,\widetilde{y},\widetilde{z}\right)=0, & \ \ \ \text{in} \ \ \ {Q},\\
		w_1=0, \ \ w_2=0 & \ \ \ \text{on} \ \ \ {\Sigma},\\
		w_1(0)=y_0-\widetilde{y}_0=w_{1,0}, \ \ w_2(0)=z_0-\widetilde{z}_0=w_{2,0}, & \ \ \ \text{in} \ \ \ (0,1),
    \end{cases}
\end{equation}

From the linearized system (\ref{eq:linearized_system1}), one has
\begin{align*}
G_1  & =w_{1,t}-b(t)\left({a}(x)w_{1,x}\right)_x-B_1 (x,t)\sqrt{a}w_{1,x}+d_{11}(x,t)w_1 + d_{12}(x,t)w_2 -\overline{h} \cara_{_{{\omega}_0}}\,,\\
 G_2 &=w_{2,t}-b(t)\left({a}(x)w_{2,x}\right)_x-B_2 (x,t)\sqrt{a}w_{2,x}+d_{21}(x,t)w_1 + d_{22}(x,t)w_2\,,
 \end{align*}
 with
 \begin{align*}
    d_{11}(x,t):=\partial_3 F_1 (x,t,\widetilde{y},\widetilde{z}), & \quad d_{12}(x,t):=\partial_4 F_1 (x,t,\widetilde{y},\widetilde{z}),\\
    d_{21}(x,t):=\partial_3 F_2 (x,t,\widetilde{y},\widetilde{z}) &\,\,\, \text{and} \,\,\, d_{22}(x,t):=\partial_4 F_2 (x,t,\widetilde{y},\widetilde{z}).
\end{align*}

We then introduce the following appropriate spaces and mappings so that the conditions of Liusternik's Theorem are satisfied:
\begin{align*}
\mathsf{V}  & :=\left\{ (w_1 ,w_2,\overline{h}) \in [L^2 (Q)]^2 \times L^{2}(\omega_0 \times (0,T)) ;\right.\\
&\qquad \left. w_i(\cdot,t) \,\,(i=1,2) \ \text{is abs. continuous in}\ (0,1),\ \text{a.e. in}\ [0, T], \right. \\
&\qquad \left. \rho_2 w_i \in L^2(Q), \quad \rho_4 \overline{h} \in L^{2}(\omega_0 \times (0,T)), \quad \rho_2 G_i \in L^2(Q), \right.\\
& \qquad\left.   w_i (1, t) \equiv w_i (0,t) \equiv 0 \ \text{a.e in}\ [0, T], \quad w_i (\cdot,0) \in
H_{a}^{1}(0,1)\right \},  \\
\end{align*}
\[
\mathsf{Z}:=\left\{  (k_1,k_2) \in [L^{2}(Q)]^2  :  \rho_{2}k_i \in L^2(Q)\,\,(i=1,2) \right\}  
\]
and
\[
\mathsf{W}:=\mathsf{Z}\times [H_{a}^{1}(0,1)]^2 .
\]
These spaces are endowed by the natural norms.

Now, let us recall the mapping $\mathcal{A}$ defined on $\mathsf{V}$ by 
$$
\mathcal{A} (w_1,w_2,\overline{h})=(\mathcal{A}_{1,1}(w_1,w_2,\overline{h}),\mathcal{A}_{1,2}(w_1,w_2,\overline{h}),w_1(\cdot,0),w_2(\cdot,0)), 
$$
where
\begin{align*}
\mathcal{A}_{1,1}(w_1,w_2,\overline{h})  & =w_{1,t}-b(t)\left({a}(x)w_{1,x}\right)_x-B_1(x,t)\sqrt{a}w_{1,x}+F_1\left(x,t,w_1+\widetilde{y},w_2+\widetilde{z}\right)\\
 & \quad -F_1\left(x,t,\widetilde{y},\widetilde{z}\right)-\overline{h}\cara_{\omega_0}\left(\frac{\ 1 \ }{\widetilde{y}} w_1 +1 \right),\\
\mathcal{A}_{1,2}(w_1,w_2,\overline{h}) &=w_{2,t}-b(t)\left({a}(x)w_{2,x}\right)_x-B_2(x,t)\sqrt{a}w_{2,x}+F_2\left(x,t,w_1+\widetilde{y},w_2+\widetilde{z}\right) \\
  & \quad  -F_2\left(x,t,\widetilde{y},\widetilde{z}\right).
\end{align*}

Note that, isolating the linear and nonlinear parts of $\mathcal{A}$, we can write
$$
\mathcal{A}(w_1,w_2,\overline{h}):=L(w_1,w_2,\overline{h}) + N(w_1,w_2,\overline{h}), \,\,\,\text{with}
$$
\begin{align*}
L(w_1,w_2,\overline{h})  & =(w_{1,t}-b(t)\left({a}(x)w_{1,x}\right)_x-B_1(x,t)\sqrt{a}w_{1,x}-\overline{h}\cara_{_{{\omega}_0}}, \\
 & \qquad w_{2,t}-b(t)\left({a}(x)w_{2,x}\right)_x-B_2(x,t)\sqrt{a}w_{2,x}, w_1(\cdot,0),w_2(\cdot,0)),\\
N(w_1,w_2,\overline{h}) &=(F_1\left(x,t,w_1+\widetilde{y},w_2+\widetilde{z}\right)-F_1\left(x,t,\widetilde{y},\widetilde{z}\right)-\overline{h}\cara_{_{{\omega}_0}} \frac{\ 1 \ }{\widetilde{y}} w_1, \\
  & \qquad  F_2\left(x,t,w_1+\widetilde{y},w_2+\widetilde{z}\right)-F_2\left(x,t,\widetilde{y},\widetilde{z}\right),0,0).
\end{align*}

Applying \emph{Liusternik’s Inverse Mapping Theorem} (see \cite{Alekseev}), we will prove that $\mathcal{A}$ has a right inverse mapping defined in a small ball contained in $\mathsf{W}$. Due to the choice of the spaces $\mathsf{V}$ and $\mathsf{W}$, the existence of the inverse mapping will imply the local null controllability of (\ref{nolineal1}). Before doing it, we will establish some results which will guarantee that $\mathcal{A}$ satisfies the hypotheses of Liusternik’s Theorem.

\begin{itemize}
    \item[(i)] The mapping $\mathcal{A}: \mathsf{V}\rightarrow \mathsf{W}$ is well defined:
\end{itemize}

Indeed, based on the definition of $\mathsf{V}, \mathsf{W}$ and from the linearity of $L$, we have that $L$ is bounded and satisfy
$$\Vert \rho_2 L(w_1,w_2,\overline{h})\Vert_{L^2(Q)}^{2} \leq C\Vert (w_1 ,w_2 ,\overline{h})\Vert_{V}^2 .$$
On the other hand, from hypothesis (\textbf {H2}), one has for $i=1,2$
        \begin{equation*}
\label{hfn}
 \begin{split}
    & \vert F_i(x,t,w_1 +\widetilde{y},w_2 +\widetilde{z})-F_i(x,t,\widetilde{y},\widetilde{z})- \partial_3 F_i(x,t,\widetilde{y},\widetilde{z})w_1 - \partial_4 F_i(x,t,\widetilde{y},\widetilde{z})w_2\vert \\
  & \qquad \leq C (\vert w_1\vert^2 + \vert w_2\vert^2).
 \end{split}
\end{equation*}

Therefore,
\begin{equation*}
		\begin{array}{lll}
			\|\rho_2N(w_1 ,w_2 ,\overline{h})\|^{2}_{L^2 (Q)}&\leq & C\displaystyle\int_{Q}\rho^{2}_{2}|w_1|^{4}dxdt + C\displaystyle\int_{Q}\rho^{2}_{2}|w_2|^{4}dxdt \\
                        && + C\displaystyle\int_{Q} \rho_{2}^2 \frac{\ 1 \ }{\widetilde{y}}\cara_{\omega_0} |\overline{h} w_1|^{2} \,dxdt \\
			&=:& C I_{1} + C I_{2} + C I_3.
		\end{array}
	\end{equation*}

Note that,
\begin{equation}
\label{I1.est}
\begin{split}
    I_{1} = & \ \iint_{Q}\rho_{2}^2  |w_1|^{4} \,dxdt\\
       \leq& \ C \int_{0}^{T}\rho_{2}^2   \Bigl[\int_{0}^1  |w_1|^{4} \,dx\Bigr] \,dt\\
    \leq& \ C \int_{0}^{T}\rho_{2}^2 \rho_{4}^{-4}\,\left[\rho_{4}^4\Vert \sqrt{a}w_{1,x} \Vert_{L^{2}(0,1)}^{4} \right]\,dt\\
        \leq & \ C\int_{0}^{T}\rho_{2}^2 \rho_{4}^{-4} \,\left[ \sup_{t \in [0,T]} \int_{0}^1 \rho_{4}^{4} \vert \sqrt{a}w_{1,x}\vert^{4}\, dx \right] \,dt\\
          \leq & \ C\Vert (G_1,G_2,w_1(\cdot,0), w_2(\cdot,0)) \Vert^{4}_{\mathsf{W}} .
 \end{split}
\end{equation}

In a similar way, $I_2 \leq \Vert (G_1,G_2,w_1(\cdot,0), w_2(\cdot,0)) \Vert^{2}_{\mathsf{W}}$.

Now, we will analyze the term involving the multiplicative control $\overline{h}w_1$
\begin{equation}
\label{I3.est}
 \begin{split}
    I_{3} \leq& \int_{Q}\rho_{2}^2 \cara_{\omega_0} |\overline{h} w_1|^{2} \,dxdt\\
    \leq& \ C \int_{0}^{T}\rho_{2}^2  \Bigl[\int_{0}^1 \cara_{\omega_0} |\overline{h}|^{2} |w_1|^{2} \,dx\Bigr] \,dt\\
    \leq& \  C \int_{0}^{T}\rho_{2}^2 \,\left[ \Vert \overline{h}\Vert_{L^{4}(0,1)}^{2} \,\Vert w_1 \Vert_{L^{4}(0,1)}^{2} \right]\,dt\\
    \leq& \  C \int_{0}^{T}\rho_{2}^2 \rho_{5}^{-2} \rho_{4}^{-2}\,\left[ \rho_{5}^{2} \Vert  \overline{h}_x\Vert_{L^{2}(0,1)}^{2} \,\rho_{4}^{2}\Vert  \sqrt{a}w_{1,x} \Vert_{L^{2}(0,1)}^{2} \right] \,dt\\
        \leq&  \ C \,\Vert \rho_{5} \,\overline{h}\Vert_{L^{\infty}(0,T,H_{a}^{1}(0,1))}^{2} \sup_{t\in [0,T]}\Vert \rho_{4} \sqrt{a}w_{1,x} \Vert_{L^{\infty}(0,T,H_{a}^{1}(0,1))}^{2}\\
    \leq&  \ C \left(\Vert \rho_2 G_1 \Vert^2_{L^2(Q)} + \Vert\rho_2 G_2 \Vert^2_{L^2(Q)} + \Vert w_{1,0}\Vert^2_{L^2 (0,1)} + \Vert w_{2,0} \Vert^2_{L^2 (0,1)}\right)\\
   =&  \ C\Vert (G_1,G_2,w_1(\cdot,0), w_2(\cdot,0)) \Vert^{2}_{\mathsf{W}} .
 \end{split}
\end{equation}
This ends the proof of item (i).

\begin{itemize}
\item[(ii)] The mapping $\mathcal{A}: \mathsf{V}\rightarrow \mathsf{W}$ continuously differentiable:
\end{itemize}

Firstly, let us prove that for any $(w_1 ,w_2 ,\overline{h}) \in \mathsf{V}$ the mapping $\mathcal{A}$ is Gateaux differentiable. To do that, consider the mapping $D\mathcal{A}: \mathsf{V} \rightarrow \mathsf{W}$ given by $$D\mathcal{A}(w_1 ,w_2 ,\overline{h})=DL(w_1 ,w_2 ,\overline{h})+DN(w_1 ,w_2 ,\overline{h});\,\,\text{for}\,\, (w_1 ,w_2 ,\overline{h}) \in \mathsf{V}.$$

From the linearity of $L$, one has $DL(w_1 ,w_2 ,\overline{h}) \rightarrow L(w_1 ,w_2 ,\overline{h})$.

Now, it remains to analyze that $N$ is Gateaux differentiable at any $(w_1 ,w_2 ,\overline{h}) \in \mathsf{V}$. Note that, for any $(w_{1}^{\ast},w_{2}^{\ast},\overline{h}^{\ast}) \in \mathsf{V}$ and $\lambda >0$, one has
\begin{equation*}
\label{DA.estim}
 \begin{split}
    & \Big\| \frac{1}{\lambda}[F_1\left(x,t,w_1+\lambda w_{1}^{\ast} +\widetilde{y},w_2 +\lambda w_{2}^{\ast} +\widetilde{z}\right)-F_1\left(x,t,w_1+\widetilde{y},w_2+\widetilde{z}\right)]  \\
    & \qquad  -\partial_3 F_1 \left(x,t,w_1 +\widetilde{y},w_2 +\widetilde{z}\right)w_1 - \partial_4 F_1 \left(x,t,w_1 +\widetilde{y},w_2 +\widetilde{z}\right)w_2\Big\|_{\mathsf{Z}}  \\
    & + \Big\| \frac{1}{\lambda}[F_2\left(x,t,w_1+\lambda w_{1}^{\ast} +\widetilde{y},w_2 +\lambda w_{2}^{\ast} +\widetilde{z}\right)-F_2\left(x,t,w_1+\widetilde{y},w_2+\widetilde{z}\right)]  \\
     & \qquad  -\partial_3 F_2 \left(x,t,w_1 +\widetilde{y},w_2 +\widetilde{z}\right)w_1 - \partial_4 F_2 \left(x,t,w_1 +\widetilde{y},w_2 +\widetilde{z}\right)w_2\Big\|_{\mathsf{Z}}  \\
    & + \Big\| \frac{1}{\lambda}[(\overline{h}+\lambda \overline{h}^{\ast})\cara_{_{{\omega}_0}} \frac{\ 1 \ }{\widetilde{y}} (w_1 +\lambda w_{1}^{\ast})-\overline{h}\cara_{_{{\omega}_0}} \frac{\ 1 \ }{\widetilde{y}} w_1]\Big\|_{\mathsf{Z}}  \\
    =& \ A_{1}+A_{2}+A_3.
 \end{split}
\end{equation*}
Following the same idea as for $I_2$ in item $\text{(i)}$ and using Mean Value Theorem, we get that $A_1$ and $A_2$ converges to zero as $\lambda \rightarrow 0$. On the other hand, clearly $A_3$ converges to zero as $\lambda \rightarrow 0$. 

So, we conclude that $\mathcal{A}$ is Gateaux differentiable.

In the sequel, we will prove that $D\mathcal{A}: \mathsf{V} \rightarrow \mathsf{W}$ is continuous. To this end, take $(w_1 ,w_2 ,\overline{h}) \in \mathsf{V}$ and let $(w_{1}^{n},w_{2}^{n},\overline{h}^n)_{n=1}^{\infty}$ be a sequence such that
$$(w_{1}^{n},w_{2}^{n},\overline{h}^n) \rightarrow (w_1 ,w_2 ,\overline{h}) \,\,\,\text{in}\,\,\,\mathsf{V}.$$

We have that
\begin{equation*}
\label{DA11.estim}
 \begin{split}
    & \Vert (D\mathcal{A}_{1,1}(w_{1}^{n},w_{2}^{n},\overline{h}^{n})-D\mathcal{A}_{1,1}(w_1,w_2,\overline{h}))(w_{1}^{*},w_{2}^{*},\overline{h}^{*})\Vert_{ \mathsf{W}}^{2}  \\
    & \quad  \leq \int_{Q} \rho_{2}^2 \vert \partial_3 F_1 (x,t,w_{1}^n +\widetilde{y},w_{2}^n +\widetilde{z})-\partial_3 F_1 (x,t,w_1 +\overline{y},w_2 +\overline{z} )\vert^2  \vert w_{1}^{\ast}\vert^2 dxdt \\
    & \qquad + \int_{Q} \rho_{2}^2 \vert \partial_4 F_1 (x,t,w_{1}^n +\widetilde{y},w_{2}^n +\widetilde{z})-\partial_4 F_1 (x,t,w_1 +\overline{y},w_2 +\overline{z} )\vert^2 \vert w_{2}^{\ast}\vert^2 dxdt  \\
           & \qquad + \int_{Q} \rho_{2}^2 \left( (\overline{h}^n -\overline{h})w_{1}^{\ast}+(w_{1}^n -w_1)\overline{h}^{\ast}\right)^2 dxdt\\
           & \quad = D_1 +D_2 + D_3.
 \end{split}
\end{equation*}

Note that, 
\begin{equation*}
\label{D1.estim}
 \begin{split}
    & D_1 \leq \sup_{(x,t) \in Q} \vert \partial_3 F_1 (x,t,w_{1}^n +\widetilde{y},w_{2}^n +\widetilde{z})-\partial_3 F_1 (x,t,w_1 +\overline{y},w_2 +\overline{z} )\vert^2\int_{Q} \rho_{2}^2 \vert w_{1}^{\ast}\vert^2 dxdt \\
    & \quad  \leq \Vert (w_{1}^n ,w_{2}^n ,\overline{h}^n)-(w_1 ,w_2,\overline{h}) \Vert_{V}^2 \Vert (w_{1}^{\ast},w_{2}^{\ast},\overline{h}^{\ast})\Vert_{V}^2 \rightarrow 0 \,\,\,\text{as} \,\,\, n\rightarrow +\infty.
 \end{split}
\end{equation*}

Analogously, $D_2 \rightarrow 0$, as $n \rightarrow +\infty$.

Finally,
\begin{equation*}
\label{D3.estim}
 \begin{split}
    & D_3 \leq 2\int_{Q} \rho_{2}^2 \left( (\overline{h}^n -\overline{h})w_{1}^{\ast}\right)^2 dxdt+ 2\int_{Q}\rho_{2}^{2}\left((w_{1}^n -w_1)\overline{h}^{\ast}\right)^2 dxdt \\
    & \quad = 2E_1 +2E_2.
 \end{split}
\end{equation*}

Next, arguing as in (\ref{I3.est}) and using Corollary \ref{rcont}, we verify that
\begin{equation*}
\label{E1.est}
 \begin{split}
    E_1 =& \int_{Q} \rho_{2}^2 \left( (\overline{h}^n -\overline{h})w_{1}^{\ast}\right)^2 dxdt\\
    \leq& \  C\Vert \rho_5(\overline{h}^n -\overline{h})\Vert_{L^{\infty}(0,T,H_{a}^{1}(0,1))}^{2}\,\sup_{t\in [0,T]}\Vert \rho_{4} \sqrt{a}w_{1,x}^{\ast} \Vert_{L^{\infty}(0,T,H_{a}^{1}(0,1))}^{2} \\
      \leq&  \ C \Vert \rho_5(\overline{h}^n -\overline{h})\Vert_{L^{\infty}(0,T,H_{a}^{1}(0,1))}^{2}\,\Vert (w_{1}^{*},w_{2}^{*},\overline{h}_{1}^{*})  \Vert _{\mathsf{V_1}}^{2}\\
    \leq&  \ C \Vert \rho_5(\overline{h}^n -\overline{h})\Vert_{L^{\infty}(0,T,H_{a}^{1}(0,1))}^{2} \rightarrow 0\\
     \end{split}
\end{equation*}

The estimate for $E_2$ is analogous to the case $E_1$. Thus, we conclude that $D\mathcal{A}_{1,1}$ is a continuous mapping. Similarly, this conclusion remains valid to $D\mathcal{A}_{1,2}$. 

Thereby, we have that $(w_1 ,w_2 ,\overline{h})\mapsto D\mathcal{A} (w_1 ,w_2 ,\overline{h})$ is continuous from $\mathsf{V}$ into $\mathcal{L}(\mathsf{V},\mathsf{W})$.

\begin{itemize}
\item[(iii)] The mapping $D\mathcal{A}(0,0):\mathsf{V}\rightarrow \mathsf{W}$ is onto:
\end{itemize}

Let us fix $(G_1,G_2,w_1(\cdot ,0),w_2(\cdot ,0)) \in \mathsf{W}$. From Theorem \ref{prop:linear_control}, we know that there exists $(w_1,w_2,\overline{h})\in \mathsf{V}$ that solves (\ref{eq:linearized_system1}). In other words,
\begin{equation*}
 \begin{split}
    D\mathcal{A}(0,0,0)(w_{1}^{*},w_{2}^{*},\overline{h}^{*})  =& (D\mathcal{A}_{1,1} (0,0,0),D\mathcal{A}_{1,2} (0,0,0),w_1(\cdot ,0),w_2(\cdot ,0))(w_{1}^{*},w_{2}^{*},\overline{h}^{*})\\
              =& (G_1,G_2,w_{1,0},w_{2,0})\,,
 \end{split}
\end{equation*}
where
\begin{align*}
D\mathcal{A}_{1,1}(0,0,0)(w_{1}^{\ast},w_{2}^{\ast},\overline{h}^{\ast})  & =w_{1,t}^{\ast}-b(t)\left({a}(x)w_{1,x}^{\ast}\right)_x-B_1 (x,t)\sqrt{a}w_{1,x}^{\ast}+d_{11}(x,t)w_{1}^{\ast}\\
 & \quad + d_{12}(x,t)w_{2}^{\ast}-\overline{h}^{\ast} \cara_{_{{\omega}_0}},\\
D\mathcal{A}_{1,2}(0,0,0)(w_{1}^{\ast},w_{2}^{\ast},\overline{h}^{\ast}) &=w_{2,t}^{\ast}-b(t)\left({a}(x)w_{2,x}^{\ast}\right)_x-B_2 (x,t)\sqrt{a}w_{2,x}^{\ast}+d_{21}(x,t)w_{1}^{\ast} \\
  & \quad  +d_{22}(x,t)w_{2}^{\ast}.
\end{align*}

It completes the proof of item (iii).

\vspace{0.3cm}

\noindent {\textbf{Proof of Theorem \ref{th1}}}

From items (i), (ii) and (iii), we have proved that $\mathcal{A}: \mathsf{V}\rightarrow \mathsf{W}$ is a continuously differentiable mapping, whose derivative $S\mathcal{A}(0,0.0)$ is onto. As a result, \emph{Liusternik's Inverse Function Theorem} can be applied in order to obtain a sufficiently small $\varepsilon >0$ and a right inverse mapping $\mathcal{A}^{-1}: B_{\varepsilon}(0) \subset \mathsf{W} \rightarrow \mathsf{V}$ of $\mathcal{A}$. Hence, taking $(w_{1,0},w_{2,0})\in [H_{a}^{1}(0,1)]^2 $ satisfying $\Vert (w_{1,0},w_{2,0}) \Vert_{[H_{a}^{1}(0,1)]^2} \leq \varepsilon$, we have that $(w_1,w_2 ,\overline{h}):=\mathcal{A}^{-1}(G_1,G_2,w_{1,0},w_{2,0})$ solves (\ref{nolineal1}) and consequently (\ref{nolineal}). Furthermore, since $\Vert \rho_{2} w_1\Vert_{L^2(0,1)}^{2} +\Vert \rho_{2} w_2\Vert_{L^2(0,1)}^{2} <+ \infty$, we get $w_1(x,T)=w_2(x,T)=0$ in $(0,1)$.

Next, taking $y=w_1+ \widetilde{y}$ and $z=w_2+ \widetilde{z}$, then  (\ref{eq:PDE}) is satisfied and $y(\cdot,T)=\widetilde{y}(\cdot, T)$ and $z(\cdot,T)=\widetilde{z}(\cdot, T)$ in $(0,1)$.
  \qed

\vspace{0.5cm}

\noindent {\textbf{Proof of Theorem \ref{th2}}}

Let $(w_{1,0},w_{2,0})\in [H_{a}^{1}(0,1)]^2 $ and consider the system (\ref{nolineal1}) with control $\overline{h} \equiv 0$. The estimate (\ref{cl2}) in Theorem \ref{twpds} guarantees the existence of $\widetilde{T}>0$ large enough such that $$\Vert w_1(\cdot ,\widetilde{T}), w_2(\cdot ,\widetilde{T})\Vert_{[H_{a}^{1}]^2 (0,1)} \leq C e^{-\mu \widetilde{T}}\,.$$

Now, suppose $\varepsilon >0$ is such that 
\begin{equation}\label{est.varepsilon}
\frac{1}{\varepsilon} \leq \frac{1}{C e^{-\mu \widetilde{T}}}
\end{equation}
and denote 
\begin{equation}\label{est.widetildeT}
(w_{1,\widetilde{T}},w_{2,\widetilde{T}}):=(w_1(\cdot ,\widetilde{T}),w_2(\cdot ,\widetilde{T}))
\end{equation}

Thereby, from (\ref{est.varepsilon}) and (\ref{est.widetildeT}), one has $\Vert (w_{1,\widetilde{T}},w_{2,\widetilde{T}}) \Vert_{[H_{a}^{1}]^2 (0,1)} \leq \varepsilon$. So, now we are in the same conditions of Theorem \ref{th1}. Then, there exists a control $\overline{h} \in L^2(\omega_0 \times (0,T))$ driving the solution $(w_1,w_2)$ from $(w_1(\cdot ,\widetilde{T}),w_2(\cdot ,\widetilde{T}))$ to $0$ in time $T>\widetilde{T}$, as desired.

Since $y=w_1+ \widetilde{y}$ and $z=w_2+ \widetilde{z}$, then  (\ref{eq:PDE}) is satisfied and $y(\cdot,T)=\widetilde{y}(\cdot, T)$ and $z(\cdot,T)=\widetilde{z}(\cdot, T)$ in $(0,1)$.
  \qed

\section{Additional comments}
\label{sec:final_remarks}
In the present work, we established the decay of solutions and analyzed the local and global exact controllability to trajectories of system (\ref{eq:PDE}) with distributed controls locally supported in space. These results can be extended to other settings. In particular, one may consider nonlinear systems of the form
\begin{equation*}\label{op1}
\left\{
\begin{array}
    [c]{lll}%
    y_t - b(t)\Big(\beta_1 (x,\displaystyle\int_{0}^{1}y )  y_{x}\Big)_{x} + B_1 (x,t) \sqrt{a} y_x +F_1(x,t,y,z)=h\cara_{\omega}y & (x,t) \in (0,1)\times (0,T),&
    \\
    z_t - b(t)\Big(\beta_2 (x,\displaystyle\int_{0}^{1}z )  z_{x}\Big)_{x} + B_2 (x,t) \sqrt{a} z_x+ F_2(x,t,y,z)=0 & (x,t) \in (0,1)\times (0,T),&
    \\
    y(0,t)=y(1,t)=z(0,t)=z(1,t)=0 & t \in (0,T), & \\
    y(x,0)=y_{0}(x), \ z(x,0)=z_{0}(x) & x \in (0,1), & 
\end{array}
\right.  %
\end{equation*}
where, for $i \in \{1,2\}$, $\beta_i$ is a separated variables function given by $\beta_i(x,r)=\mu_i (r)a(x)$ such that $\mu_i:\mathbb{R} \rightarrow \mathbb{R}$ is a $C^1$ function with bounded derivative. The function $\beta_i$ defines an operator which degenerates at $x=0$ and has a nonlocal term. More precisely, the function $a$ behaves $x^{\alpha}$, with $\alpha \in (0,1)$.

  An interesting case deals with the controllability of one-phase Stefan-like problems with the following structure:
  \begin{equation*}\label{op2}
\left\{
\begin{array}
    [c]{lll}%
    y_t - b(t)\left(a(x) y_{x}\right)_{x} + B(x,t) \sqrt{a} y_x + F(x,t,y)=h\cara_{\omega}y, &  &
    (x,t) \in Q_L ,\\
    y(0,t)=y(L(t),t)=0, &  & t \in (0,T),\\
    y(x,0)=y_{0}(x), &  & (x,t) \in (0,L_0).
\end{array}
\right.  %
\end{equation*}
$$
b(t)a(x)y_{x}(L(t),t)=-L'(t),\quad t \in (0,T),
$$
where $Q_L$ stands for the following set:
$$
Q_L=\{ (x,t);\,\, x\in (0,L(t)),\,\,t \in (0,T)\},
$$
with $0<L_0<L(t)<D,\,\,t \in (0,T)$.

\vspace{0.3cm}
Some of these extensions will be considered in the near future.

\vspace{0.8cm}

\noindent\textbf{Acknowledgments}

This study was financed in part by the Coordenação de Aperfeiçoamento de Pessoal de Nível Superior - Brasil (CAPES) - Finance Code 001.
A.S.G and L.Y. were partially supported by CAPES-Brazil.

\bibliographystyle{abbrv}
\bibliography{referencias}

\end{document}